\documentclass[a4paper]{article}

\usepackage[latin1]{inputenc}
\usepackage{latexsym}
\usepackage{amsmath,amssymb,amsfonts,amscd,amsthm}
\usepackage{MnSymbol,stmaryrd}
\usepackage{enumerate}
\usepackage{hyperref}
\usepackage[all,cmtip]{xy}
\usepackage[titletoc,title]{appendix}
 \usepackage[dvips]{graphicx}
 \usepackage{relsize}

\newcommand{\Ad}{\operatorname{Ad}}

\newcommand{\tr}{\operatorname{tr}}

\newcommand{\cs}{\operatorname{cs}}
\newcommand{\vol}{\operatorname{vol}}

\newcommand{\CC}{\mathbb{C}}
\newcommand{\EE}{\mathrm{E}}
\newcommand{\FF}{\mathrm{F}}

\newcommand{\NN}{\mathbb{N}}
\newcommand{\QQ}{\mathbb{Q}}
\newcommand{\RR}{\mathbb{R}}
\newcommand{\ZZ}{\mathbb{Z}}
\newcommand{\TT}{\mathbb{T}}
\newcommand{\T}{\mathcal{T}}

\newcommand{\SU}[1][2]{\mathrm{SU} (#1)}

\newcommand{\PSL}{\mathrm{PSL}}
\newcommand{\SL}{\mathrm{SL}}

\newcommand{\tor}{\operatorname{tor}}

\newtheorem{thm}{Theorem}[section]
\newtheorem{prop}[thm]{Proposition}
\newtheorem{lemma}[thm]{Lemma}

\newtheorem{cor}[thm]{Corollary}
\newtheorem{conj}[thm]{Conjecture}

\theoremstyle{definition}
       \newtheorem{definition}[thm]{Definition}
       \newtheorem{remark}[thm]{Remark}
       \newtheorem{example}[thm]{Example}
       \newtheorem{question}[thm]{Question}
       \newtheorem{assumption}[thm]{Assumption}

\usepackage{hyperref}

\title{Reidemeister torsion, hyperbolic three-manifolds, and character varieties}

\author{Joan Porti\thanks{Partially supported
by Mineco through grant MTM2012-34834}}

\date{\today}

\begin{document}

\maketitle

\begin{abstract}
This is a survey on Reidemeister torsion for hyperbolic three-manifolds of finite volume. 
Torsions are viewed as topological invariants and
also as functions on the variety of representations in $\SL_2(\CC)$. In both cases,
the torsions may also be computed after composing with
 finite dimensional representations 
of $\SL_2(\CC)$. 
In addition the paper  deals with the torsion of the adjoint representation as a function on the variety of $\PSL_{n+1}(\CC)$-characters,
using that the first cohomology group with coefficients twisted by the adjoint is the tangent space to the variety of characters.

\end{abstract}

\begin{footnotesize}
\tableofcontents
 
\end{footnotesize}

\section{Introduction}
\label{sec:intro}

The goal of this paper is to survey some results on Reidemeister torsions of orientable, hyperbolic
three-manifolds of finite volume. Torsions are viewed  as invariants of  hyperbolic manifolds and also
as functions on the variety of 
$\SL_2(\CC)$-characters. 

The paper uses combinatorial torsion, though some relevant  developments
described here are proved using analytic torsion, which is briefly mentioned. 
No details on other kinds of torsion are provided, like $L^2$-torsion.
There are remarkable recent surveys on twisted Alexander polynomials \cite{Morifuji15,FV11}
and on abelian torsions \cite{Massuyeau}.
There are also the classical surveys of Milnor \cite{Milnor68} and Turaev~\cite{Turaev86}, as well as the  books by Turaev \cite{Turaev01,Turaev02}
and Nicolaescu \cite{Nicolaescu}.

\medskip

Besides Section~\ref{sec:combandanl} (devoted to general tools on torsions),
only  orientable hyperbolic  three-manifolds $M^3$  of finite volume are considered.
The hyperbolic structure is unique, by Mostow-Prasad theorem, so their holonomy is unique up to conjugacy. It  lifts to a representation 
in $\SL_2(\CC)$. The lift is naturally associated to a spin structure $\sigma$ 
and it is acyclic, by a theorem of Raghunathan. 
Thus this yields  a topological invariant of the oriented manifold with a spin 
structure (Definition~\ref{defn:tor2closed} in the closed case
and Definition~\ref{defn:tor2cusped} in the non-compact case):
$$
\tau(M^3, \sigma)\in \CC^*.
$$
Its main properties are described, in  particular, its behavior by Dehn filling 
allows to construct sequences of closed manifolds whose volume stays bounded 
but 
the torsion converges to infinity (Corollary~\ref{cor:dense}). 
This must be compared with Theorem~\ref{thm:BV} (due to  Bergeron and Venkatesh) 
on the asymptotic behavior
of torsions by coverings, that yields  sequences of coverings $M^3_ n\to M^3 $ 
so that 
$
\left|\tau(M^3_ n, \sigma)\right|
$ 
grows with the exponential of the volume of $M_n^3$.
For those sequences the injectivity radius converges to infinity. Some results   
suggest that short geodesics may play a role in the behavior of this torsion, 
for instance  the  surgery formula,
Proposition~\ref{prop:SF}.

Additionally, the paper deals with finite dimensional  irreducible 
representations of $\SL_2(\CC)$. Their  composition with the lift of the 
holonomy
provides a family of invariants of a
closed hyperbolic manifold, oriented and with a spin structure. 
A remarkable theorem of M\"uller relates the asymptotic behavior of those 
invariants with the volume (Theorem~\ref{thm:Muller}).

\medskip

For simplicity, 
 finite volume hyperbolic three-manifolds  are assumed to have
 a single cusp.
I am  interested in the torsion as a function on the variety of characters. The 
distinguished component
 is the component of the variety of  $ \SL_2(\CC)$-characters that contains a 
lift of the holonomy of the complete structure.
It is easy to see that the torsion is a rational (meromorphic) function on this 
curve. It is shown here that  for a knot 
exterior (more generally,  for a manifold with first Betti number 1), this 
torsion is a regular function (it has no poles)
on the distinguished component.

The paper also discusses the functions obtained by composing the representation 
with the irreducible representations of $\SL_2(\CC)$.
In particular the torsion for the adjoint representation occurs in the volume 
conjecture, which is also quickly mentioned but not  analyzed.
\medskip

Finally I describe the torsion of the adjoint representation as a function on 
the variety of characters in $\PSL_{n+1}(\CC)$.
I considered the case  for  $\PSL_{2}(\CC)$ in \cite{Memoirs} and 
Kitayama and Terashima discuss the general case
for $\PSL_{n+1}(\CC)$  in
\cite{KitayamaTerashima15}. The relevant fact for this torsion is a result of 
Andr\' e Weil, that identifies the tangent space to the variety of characters
with the first cohomology group with coefficients twisted by the adjoint 
representation (under generic hypothesis). 
In this way this torsion is related to local parameterizations of the 
deformation space, which amounts to choose peripheral curves. In particular,
a formula for the change of curve is provided, and this allows to define a 
volume form under some circumstances (not at the holonomy of the complete 
structure,
but for instance for  characters in $\operatorname{SU}(n)$), as done by Witten 
for surfaces and Dubois for knot exteriors and  $\operatorname{SU}(2)$.
In addition Weil's interpretation allows to compute the torsion for surface 
bundles from the tangent map of the monodromy on the deformation space of the 
fibre.

\medskip

The paper is organized as follows. Section~\ref{sec:combandanl} is devoted to 
the preliminaries on combinatorial torsion, including examples  
as Seifert fibered manifolds, Witten's theorem  on the volume form on 
representations of surfaces, and Johnson's construction of an 
analog to Casson's invariant. This section concludes with a brief recall of 
analytic torsion and Cheeger-M\"uller theorem on the equivalence of both.
Section~\ref{sec:torsionSL2} discusses the torsion of a hyperbolic 
three-manifold for the lift of the holonomy in $\SL_2(\CC)$. It addresses 
first closed manifolds and then the cusped ones, considering both the invariant 
and the function on the distinguished component of the 
variety of characters.
Section~\ref{sec:repsandtors} deals with the analog to the previous section, but 
instead of the torsion of representations in $\SL_2(\CC)$,
it considers   torsions of compositions with finite dimensional representations 
of $\SL_2(\CC)$. In particular it mentions the torsion of the adjoint 
representation as part of the volume conjecture.
Section~\ref{sec:adj} deals with the the torsion of the adjoint on the variety 
of characters in $\PSL_{n+1}(\CC)$, in particular recalling the work 
of Kitayama and Terashima  and the author.

The paper concludes with two appendices. Appendix~\ref{sec:notapprox} is devoted 
to the proof of a technical result: the trivial representation does not lie in 
the distinguished component of 
the variety of characters if the first Betti number is one. This is used in 
Section~\ref{sec:torsionSL2} to prove that the torsion in $\SL_2(\CC)$ is a 
regular function
without poles. Appendix~\ref{sec:cohomology} provides  results in cohomology 
that are required for the torsions of Sections~\ref{sec:torsionSL2} and 
\ref{sec:repsandtors}. It is based mainly on a vanishing theorem of Raghunathan 
and it also discusses bases for the cohomology groups.

\paragraph{Acknowledgement}
 I am grateful to Teruaki Kitano, Takahiro Kitayama,  and the anonymous referee 
for useful remarks.

\section{Combinatorial and analytic torsions}
\label{sec:combandanl}

This section overviews the definition and main properties of combinatorial 
torsion, 
as they are used later in some proofs. It also recalls briefly the definition 
of analytic torsion, as it is used in many relevant results.

\subsection{Combinatorial torsion}
\label{sec:comb}

For a detailed definition of combinatorial torsion see  \cite{Turaev86, 
Milnor66,Massuyeau, Nicolaescu}. This section follows mainly  \cite{Memoirs}, 
in particular its convention with the power $(-1)^{i+1}$ instead of $(-1)^{i}$ 
for the alternated product.
(See Remark~\ref{rem:convention} on this convention.)  The definition is given 
for chain and for cochain complexes, in a way that both homology and cohomology
give the 
same definition of the torsion of a manifold.

\subsubsection{Torsion of a chain complex}
Let $F$ be a field and  $C_*=(C_*,\partial)$ a chain complex of finite 
dimensional $F$-vector spaces:
$$
C_d\xrightarrow{\partial} C_{d-1} \xrightarrow{\partial}\cdots 
\xrightarrow{\partial} C_0.
$$
The subspaces of boundaries and cycles are denoted by 
$B_i=\operatorname{Im}(C_{i+1}
\xrightarrow{\partial} C_i)$ and $Z_i=\ker(C_{i}
\xrightarrow{\partial} C_{i-1} )$ respectively,  the homology is denoted by   by 
$H_i=Z_i/B_i$.
Assume 
$$
c_i=\{c_{i,1},\ldots,c_{i,j_i}\}
$$
is an $F$-basis for $C_i$ and 
$$h_i=\{h_{i,1},\ldots,h_{i,r_i}\},$$
is a basis for $H_i$, if nonzero. For the definition of torsion,  a basis 
  $$b_i=\{b_{i,1},\ldots,b_{i,r_i} \}$$ for $B_i$ is also required. Using the 
exact sequences:
\begin{gather}
\label{eqn:cycle} 0\to Z_i\to C_i \xrightarrow{\partial} B_{i-1}\to 0 \\
\label{eqn:homology}  0\to B_i\to Z_i\to H_i\to 0,
\end{gather}
 lift $b_{i-1}$ to $\tilde b_{i-1}\subset C_i$ in \eqref{eqn:cycle} and $h_i$ to 
$\tilde h_i\subset Z_i\subset C_ i$ 
in \eqref{eqn:homology} and  construct a new basis for $C_i$: 
\begin{equation}
b_i\sqcup \tilde  b_{i-1}\sqcup \tilde h_i ,
\end{equation}
where $\sqcup$ denotes the disjoint union. 
The bases are compared with
the determinant of the corresponding matrix.
Given two bases $\alpha=\{\alpha_1,\ldots,\alpha_r\}$ and 
$\beta=\{\beta_1,\ldots,\beta_r\}$ for $F^r$, if  $(\eta_{ij})\in 
\operatorname{M}_n(F)$
is the matrix that relates the bases, i.e.~if $\alpha_i= \sum_j 
\eta_{ij}\beta_j$,  set 
\begin{equation}
\label{eqn:det}
[\alpha,\beta]=\det(\eta_{ij})\in F^*.
\end{equation}

\begin{definition}
\label{def:torsionchain}
 The torsion of the chain complex $C_*$ with bases $\{c_i\}$ and bases 
$\{h_i\}$ 
for $H_i$ is:
 \begin{equation}
 \label{eqn:torcplx}
  \tor(C_*,\{c_i\},\{h_i\})= \prod_{i=0}^{d} [ b^i\sqcup \tilde  b_{i-1}\sqcup 
\tilde h_i, c_i] ^{(-1)^{i+1}}\in F^*/\{\pm 1\}.
 \end{equation}
\end{definition}

As it is an alternated product, it is easy to see that it does not depend on the 
choice of $b_i$, the basis for $B_i$, and it is also straightforward
that it does not depend on the lifts $\tilde b_{i-1}$ and $\tilde h_{i}$. 

\begin{remark}
\label{rem:convention}
In the alternated product defining the torsion of a complex \eqref{eqn:torcplx}, 
may authors use ${(-1)^{i}}$
instead  of  ${(-1)^{i+1}}$.
Definition~\ref{def:torsionchain} follows the convention of   
\cite{Turaev01,Turaev02,Turaev86,Milnor62,Kitano96}  for instance,
but opposite to \cite{Milnor66,Dubois05} and all papers on analytic torsion 
\cite{Cheeger79,Muller93}
and quantum invariants \cite{Witten89,BNW91, DK07,GM08,DG11, HM11,Murakami13, 
DG13, OT15, DuboisG16}.
I followed this convention in  \cite{Memoirs} but not in \cite{MFP14}.
This convention is better suited for polynomials, also for functions on the 
deformation space; the opposite one
is more useful for the interpretation of the torsion
as a volume form. 
 
We follow a different definition for the torsion of a cocomplex, which gives a 
nonstandard statement of 
Lemma~\ref{lemma:dual} (in the standard version it is corrected by a power  
$(-1)^n$).
\end{remark}

There is a formula of change of bases.
For different choices $c'_i$ and $h'_i$ of bases for $C_i$ and $H_i$ we have
\begin{equation}
\label{eqn:change}
 \frac{ \tor(C_*,\{c_i'\},\{h_i'\} )}{ \tor(C_*,\{c_i\},\{h_i\}) } =
 \prod_{i=0}^{d}  \left(\frac{[c'_i,c_i ]}{ [h'_i,h_i ]}\right)^{(-1)^i}.
\end{equation}
The proof is straightforward, see 
\cite{Milnor66,Memoirs,Nicolaescu} for details.

 \medskip

\begin{definition}
 \label{def:cochain}
Let $C^*=(C^*,\delta)$  be a \emph{cochain complex} with bases $\{c^i\}$, and 
cohomology bases $\{h^i\}$, by constructing the $b^i$ in a similar way, 
its torsion is defined as
\begin{equation}
 \label{eqn:cohom}
 \tor(C^*,\{c^i\},\{h^i\})=\prod_{i=0}^{d} [ b^i\sqcup \tilde  b^{i-1}\sqcup 
\tilde h^i, c_i] ^{(-1)^{i}}\in F^*/\{\pm 1\}.
\end{equation}
 \end{definition}

The convention of powers $(-1)^{i}$ and $(-1)^{i+1}$ has been changed in this 
definition, for the purpose of Lemma~\ref{lemma:dual}.

\begin{lemma}
Let  $(C_i)^ \star=\hom_F(C_i,F)$ be the dual cocomplex, with  coboundary 
$\delta: 
C_i^\star\to C^\star_{i-1}$ defined 
by $\delta(\theta)=\theta\circ d$. If $(h_i)^\star$ is dual to $h_i$, then 
\label{lemma:dual} 
 $$
 \tor(C_*,\{h_i\}) =\tor ((C_*)^\star, \{ (h_i)^\star\}).
 $$
\end{lemma}

Notice that this lemma  uses Definition~\ref{def:cochain} of a cochain complex. 
If 
instead of the cocomplex one considers the dual complex, 
one must  re-index the dimension $i$ by $d-i$, then the torsion of the complex 
is replaced by its $(-1)^{d+1}$-power!
Then, one may see \cite{Franz37,Milnor62} for a proof with this version of 
Lemma~\ref{lemma:dual}.

\subsubsection{Twisted chain complexes}

Let $K$ be a finite CW-complex. This paper is mostly interested in 3-manifolds 
but also in surfaces and in $S^1$, for which there is a canonical 
choice of PL-structure.
Let  
$$
\rho \colon \pi_1 K \to \mathrm{SL}_n(F)
$$  be a representation of its fundamental group.
Consider the chain complex of vector spaces
$$
C_*(K;\rho):= F^n\otimes_{\rho} C_* (\widetilde K;\ZZ)  
$$
where  $C_* (\widetilde K;\ZZ )$ denotes the simplicial complex of the universal 
covering and 
$\otimes_{\rho}$ 
means that one takes the quotient of 
$
F^n\otimes_\ZZ C_* (\widetilde K;\ZZ) $ by the $\ZZ$-module generated by 
$$\rho(\gamma)^t v \otimes c- v\otimes \gamma\cdot c  ,
$$
where $v\in F$, $\gamma\in\pi_1 K$ and $c\in C_* (\widetilde K;\ZZ)$, and $^t$ 
stands for transpose.
Namely
$$v\otimes \gamma\cdot c= \rho(\gamma)^t v \otimes c \qquad \forall 
\gamma\in\pi_1K.$$
Instead of the transpose, one could use the inverse. For some representations 
the inverse  and transpose are the same or conjugate, but this is not true in 
general,
and here this is relevant for duality with cohomology.

The boundary operator is defined by linearity and $\partial (v\otimes 
c)=v\otimes\partial c$,  for $v\in F$ and $c\in C_* (\widetilde K;\ZZ)$.
The homology of this complex is denoted by 
$$
H_*(K; \rho).
$$
Analogously, one considers the \emph{cocomplex} of cochains 
$$
C_*(K;\rho):= \hom_{\pi_1 K} ( C_* (\widetilde K;\ZZ), F^n)
$$
that has a natural coboundary operator to define the cohomology
$$
H^*(K; \rho).
$$
Then $C_*(K;\rho)$ is
 a cocomplex of finite dimensional $F$-vector spaces.
Choose $\{v_1,\ldots,v_n\}$ a $F$-basis for $F^n$ and let  $\{ e^i_1,\ldots, 
e^i_{j_i}\}$
denote the set of $i$-cells of $K$. Then 
$c_i=\{v_r\otimes\tilde e^i_s\mid r\leq n, s\leq j_i \}$ is a $F$-basis for 
$C_i(K;\rho)$.

Let $h_i$ be a basis for $H_i(K; \rho)$. One can now define the torsion by means 
of chain complexes:

\begin{definition} The torsion of $(K,\rho, \{h_i\})$ is
$$
\tor(K,\rho, \{h_i\}) = \tor( C_*(K;\rho), \{c_i\}, \{h_i\})\in F^*/\{\pm 1\}. 
$$
\end{definition}

\begin{remark}
\label{rem:independent}
\begin{enumerate}[(a)]
     \item This torsion does not depend on the \emph{lifts of the cells} $\tilde 
e_i$ nor the basis  $\{v_1,\ldots,v_n\}$ of $F^n$.

\item It does not depend either on the \emph{conjugacy class
of} $\rho$, taking care that the bases for the homology are in correspondence 
via the natural isomorphism between
the homology groups of the conjugate representations.
\end{enumerate}
\end{remark}

\begin{remark}[Sign indeterminacy]
\label{rem:sign}
The torsion lies in $F^*/\{\pm 1\}$, but there are ways to avoid the sign 
indeterminacy:
\begin{enumerate}[(a)]
 \item  When both the Euler characteristic $\chi(K)$ and   $n=\dim\rho$ are 
even, then the sign of this torsion is also well defined, 
 i.e.\ it lives
 in $F^* $.
 
  \item
When $\chi(K)$ is even but $n=\dim\rho$ is odd, the ordering of the cells is 
relevant: If  two cells are permuted in the construction,
then the sign of the torsion is changed. 
To overcome this and to get an invariant in $F^*$, Turaev \cite{Turaev86} 
noticed that given an ordering of the basis in homology with 
constant real coefficients, $H_*(K;\mathbb R)$, there is a natural way to order 
the cells of $K$, (so that the torsion with trivial coefficients is then 
positive).
 So if one has an orientation of the the homology  $H_*(K;\mathbb R)$, then  
there is a choice for the sign of the torsion.
\end{enumerate}
\end{remark}

\begin{remark}
\label{rem:topinvariant}
It is a topological invariant but not in its relative version, 
cf~\cite[Remark~2.12]{Nicolaescu}. Namely: 
\begin{enumerate}[(a)]
 \item 
It is an invariant of the simple homotopy type of $K$. Hence, by Chapman's 
theorem~\cite{Chapman74},
$\tor(K,\rho, \{h_i\}) $ is invariant by homeomorphisms.

Working with manifolds of dimension $\leq 3$, uniqueness of the  triangulation 
up to subdivision  is an alternative to Chapman's theorem.

\item 
However there is a well defined notion of torsion of a pair of CW-complexes 
$(K,L)$, but then the torsion is \emph{not a topological
invariant} of a pair. In fact Milnor used
Reidemeister torsion of a pair to distinguish two homeomorphic simplicial 
complexes that are not combinatorially equivalent \cite{Milnor61}. 
\end{enumerate}

\end{remark}

We deal now with  the construction using cohomology. Consider elements $(\tilde 
e^i_r)^\star\otimes v\colon C_i(\tilde K;\ZZ)\to F ^n$
as the morphism of $\pi_1K$-modules defined by 
$$
\left( (\tilde e^i_r)^\star \otimes v \right) (\tilde e^i_s)=\left\{
\begin{array}{ll}
  v & \textrm{ if } s=r, \\
  0 & \textrm{ if } s \neq r.
\end{array}
\right.
$$
Then $c^i=\{(\tilde e^i_r)^\star \otimes v_j\}_{j\leq  n,\ r\leq n_i}$ is a 
basis for $C^i(K; \rho)$. 
Using this basis, one may define the torsion of the  cocomplex $C_*(K;\rho)$, 
$\{c^i\}$, and $\{h^i\}$
a basis in cohomology. 
Remarks~\ref{rem:independent}, \ref{rem:sign}, and \ref{rem:topinvariant}, also 
hold true for the torsion defined from cohomology.
By using Lemma~\ref{lemma:dual}, 
we have:

\begin{remark}
\label{rem:cohomology}
The complexes $C^*(K;\rho)$  and $C_*(K;\rho)$ are dual. In addition 
$$
\tor( C^*(K;\rho), \{c^i\}, \{(h_i)^\star\})= 
\tor( C_*(K;\rho), \{c_i\}, \{h_i\}).
$$ 
Hence  both homology or cohomology can be used to define the torsion.
\end{remark}

Here Poincar\' e duality and duality between homology and cohomology is not 
discussed
even if it is used, see \cite{Memoirs} for instance.

\subsection{Geometric properties of combinatorial torsion}

This subsection recalls the basic properties of combinatorial torsion, the main 
one being Mayer-Vietoris, useful for cut and paste.
Bundles over $S^1$ are also considered, in particular instead of just the 
torsion it is more convenient to consider
the twisted Alexander polynomial, which is a torsion by \cite{Kitano96}. The 
subsection finishes with  examples.

\subsubsection{Mayer-Vietoris}
Let $K$ be a CW-complex, with subcomplexes $K_1,\ K_2\subset K$ so that 
$K=K_1\cup K_2$. Let $\rho\colon\pi_1K\to \SL_n(F)$ be a representation.
 Consider the diagram 
 of inclusions
 \[
\xymatrix{
K_1\cap K_2
\ar[d]^{i_2}
\ar[r]^{i_1} & K_1 \ar[d]^{j_1} \\ 
K_2 \ar[r]^{j_2}  & K.
}
\]
There is a Mayer-Vietoris exact sequence in homology with twisted coefficients, 
if the representations
 on $\pi_1K_1$, $\pi_1K_ 2$ and $\pi_1L$ are the restrictions:
 \begin{multline}
 \label{eqn:MV}
  \cdots  \to \bigoplus_L H_i(L;\rho) \xrightarrow{i_{1*}\oplus i_{2*}}  
H_i(K_1;\rho)\oplus H_i( K_2;\rho)
  \xrightarrow{j_{1*}-j_{2*}}  H_i(K;\rho)  \\
 \to \bigoplus_L  H_{i-1}(L;\rho)\to\cdots
 \end{multline}
 where $L$ runs on the connected components of  $K_1\cap K_2$. Notice that 
different choices of base-points yield canonical isomorphisms
 between the homology groups, hence we can consider  $K_1\cap K_2$ not 
connected, 
 
Choose a basis for each of these cohomology groups  $h_*$ for $H_*(K;\rho)$, 
$h_{1*}$ for $H_*(K_1;\rho)$,
 $h_{2*}$ for $H_*(K_2;\rho)$, and $h_{L*}$ for $H_*(L ;\rho)$. The long exact 
sequence 
 \eqref{eqn:MV} is viewed as a complex. Its torsion is denoted by $\tor 
(\mathcal H, h_{**})$.

\begin{prop}[Mayer-Vietoris] 
\label{prop:MV}
 Let $K$ be a CW-complex, with subcomplexes $K_1\cap K_2$ so that $K=K_1\cup 
K_2$. Let $\rho\colon\pi_1K\to \SL_n(F)$ be a representation and choose basis 
in 
homology.
Then
 $$
 \tor (K, h_*;\rho)=\frac{\tor (K_1,h_{1*};\rho)\tor (K_2,h_{2*};\rho) 
}{\prod\limits_L\tor (L,h_{L*};\rho) \tor (\mathcal H, h_{**}) }
 $$
 where $L$ runs on the connected components of $K_1\cap K_2$.
 \end{prop}

 This is used in surgery formulas (e.g.\ Proposition~\ref{prop:SF}) or for the 
mapping torus (Proposition~\ref{prop:MapTorus}). 
 The proof can be found in \cite{Milnor66} or  in \cite[Section 0.4]{Memoirs}.

 \subsubsection{Polynomials}
 \label{subsection:polynomial}
 
Here some properties of twisted Alexander polynomials viewed as torsions are 
briefly discussed, 
since they are  convenient for describing some results of Reidemeister torsion. 
See the recent surveys
\cite{Morifuji15,FV11} fore more details on twisted Alexander polynomials.
 
Start with a surjective morphism $\phi\colon\pi_1 K\to \ZZ$. Instead of a 
representation $\rho\colon\pi_1K\to \SL_n(F)$, consider the twisted 
representation
$$
 \rho\otimes \phi\colon\pi_1 K\to \operatorname{GL}_n(F(t))
$$
where $F(t)$ is the field of fractions of the polynomial ring $F[t]$.
 
If $H_*(K;\rho)=0$, then $H_*(K;\rho\otimes \phi)=0$ \cite{Milnor68,Kitano96} 
and there is a well defined torsion, or twisted polynomial:
$$
\Delta_{K,\rho,\phi}(t)=\tor (K,\rho,\phi)\in F(t)/\pm t^{n\, \ZZ}.
$$
As the determinant is not one, there is an indeterminacy factor  $t^{k\, n}$, 
for some integer $k\in\ZZ$.

\begin{remark}
\label{rem:evaluation}
Assume $H_*(K;\rho)=0$. Then 
$$ \Delta_{K,\rho,\phi}(1)= \tor (K,\rho).$$
 This is proved from the map at the chain level $C_*(K;\rho\otimes\phi)\to 
C_*(K;\rho)$
 induced by evaluation $t=1$.
\end{remark}

 If  $\Delta_{K,\rho,\phi}=\sum_i a_i t^i$, then its degree can be defined as:
 $$
 \deg ( \Delta_{K,\rho,\phi} )   =\max \{i-j\mid a_i,a_j\neq 0\}.
 $$
 
\begin{prop}
\label{prop:degree}
If $M^3$ is a three-manifold and  $\Sigma\subset M^3$ is a surface dual to 
$\phi$ such that 
$H_0(\Sigma;\rho)=0$, then 
 $$
 \deg \Delta_{M^3,\rho,\phi} \leq -\chi(\Sigma) n.
 $$
\end{prop}

This proposition  can be proved using Mayer-Vietoris to
a tubular neighborhood of the surface and its exterior.

\medskip

For a knot ${\mathcal K}\subset S^3$, there is a natural surjection 
$\phi\colon\pi_1(S^3-{\mathcal K})\to \ZZ$.
Let $A_{\mathcal K}$ denote its Alexander polynomial. 
In \cite{Lin01} (this reference is based on a preprint from 1990) Lin defined a 
twisted Alexander polynomial $A_{\mathcal K}^\rho$ that lives in $F(t)$,
that was later modified by Wada
\cite{Wada94}.

Milnor in 1962 for the untwisted Alexander polynomials,  and Kitano in 1996 for 
the twisted ones, proved:

\begin{thm}
 For a knot exterior $S^3-{\mathcal K}$, we have:
 \begin{enumerate}[(a)]
  \item  \cite{Milnor62} For  the trivial representation:
 $$
\Delta_{S^3-\mathcal K,1,\phi} =\tor(S^3-{\mathcal K}, \phi)= \frac{A_{\mathcal 
K}(t)}{(t-1)}.
 $$
\item \cite{Kitano96} For $\rho$ a non-trivial acyclic representation:
$$
\Delta_{S^3-\mathcal K,\rho,\phi}=
\tor(S^3-{\mathcal K},\rho\otimes \phi)= A_{\mathcal K}^\rho(t).
$$
 \end{enumerate}
\end{thm}

Milnor's formula, for the untwisted polynomial, has been generalized to links 
and to other three-manifolds  in 1986 by Turaev  \cite{Turaev86}. See also 
\cite{FV11, Morifuji15}.

The following is straightforward from Mayer-Vietoris, Proposition~\ref{prop:MV}:
 
\begin{prop}
\label{prop:MapTorus}
 Let $K$ be a CW complex, $f\colon \vert K\vert \to \vert K\vert $ a 
homeomorphism 
with mapping torus $M_f$. Then
 $$\Delta _{M_f,\rho,\phi}=\prod_{i=0}^{\dim K} \det (f_i - 
t\operatorname{Id})^{(-1)^{i+1}}\, ,  $$
where $f_i\colon H_i(K;\rho)\to H_i(K;\rho)$ denotes the induced map in 
homology.
 In particular, by Remark~\ref{rem:evaluation},
 $$
 \tor(M_f,\rho)=\prod_{i=0}^{\dim K} \det (f_i - 
\operatorname{Id})^{(-1)^{i+1}}. 
 $$
\end{prop}

Another relevant issue for this polynomial is Turaev's interpretation of the 
twisted Alexander polynomial: namely this polynomial encodes the torsion of 
the finite cyclic coverings (\cite{Turaev86}, see 
also~\cite{DY12,Porti04,SW02,Raimbault12, Yamaguchi12}).

\begin{prop}
\label{prop:Fox}
Let $K$, $\rho\colon\pi_1 K\to \SL_n(F)$, and $\phi\colon\pi_1 K\to \ZZ$ be as 
above.
Let $K_m\to K$ be the cyclic covering of order $m$ corresponding to the kernel 
of $\phi$ composed with the projection 
 $\ZZ\to\ZZ/m\ZZ$.
Then
\begin{enumerate}[(a)]
 \item $H_*(K_m;\rho)=0$ if and only if $\Delta_{M^3,\rho,\phi} (\zeta)\neq 0$ 
for every $\zeta\in\CC$ 
 satisfying $\zeta^m=1$.
 \item If  $H_*(K_m;\rho)=0$ then 
 $$
 \tor(K_m, \rho)=\prod\limits_{i=0}^{m-1} \Delta_{M^3,\rho,\phi} ( \zeta^i)
 $$ 
 for $\zeta\in\CC$ a primitive $m$-root of the unity.
\end{enumerate}
\end{prop}

\subsubsection{Examples}

The first  examples below can be found in many references, for instance 
\cite[Chapter~2]{Nicolaescu}.

\begin{example}[The circle]
\label{ex:circle}
Consider the circle $S^1$. A representation $\rho$ of its fundamental group 
$\pi_1 S^1\cong\ZZ$ is determined by the image of its generator, 
that we denote by $A\in\SL_n(F)$.
Notice that  the homology and cohomology of $S^1$ twisted by $\rho$ is 
determined by $H^0(S^1;\rho)$,
by duality between homology and cohomology and vanishing of the Euler 
characteristic.
Since  $H^0(S^1;\rho)$ is isomorphic to the subspace of invariant elements 
$(F^n)^{\rho(\ZZ)}=\ker (A-\operatorname{Id} ) $,
\begin{equation*}
H_*(S^1;\rho)=0\textrm{ if and only if }\det (A-\operatorname{Id})\neq 0.
\end{equation*}
On the other hand, $S^1$ is just the mapping torus of the identity map on the 
point $*$. The homology of the point is $H_0(*;\rho)\cong F^n$ and the action 
of 
the 
return map is multiplication by $A$.
Thus, by Proposition~\ref{prop:MapTorus}:
$$
\Delta(S^1,\rho\otimes\phi)=\frac{1}{\det (A-t\operatorname{Id})} \qquad\textrm{ 
and }\qquad \tor(S^1, \rho)=\frac{1}{\det (A-\operatorname{Id})}.
$$
In fact using  Proposition~\ref{prop:MapTorus} is a fancy way of computing the 
torsion of the circle. It is more natural to view 
 Proposition~\ref{prop:MapTorus} as a generalization of the torsion of the 
circle.
\end{example}

\begin{example}[The 2-torus]
\label{ex:2torus}
The homology  and the cohomology of the two-torus $T^2$,  $H_*(T^2;\rho)$  and 
$H^*(T^2;\rho)$, are
determined by $H^0(T^2;\rho)$ which is the subspace of $F^n$ of invariant 
elements. This assertion follows from the different dualities (Poincar\' e, and 
homology/cohomology)
and the Euler characteristic. Thus if $F^n$ has no nonzero invariant elements by 
$\rho(\pi_1T^2)$, then $H_ *(T^2;\rho)=0$. In this case
$$
\tor(T^2;\rho)=1.
$$
This can be proved viewing $T^2$ as the mapping torus of the identity on $S^1$ 
and Proposition~\ref{prop:MapTorus},
as the action on $H^1(S^1;\rho)$ is the same as on  $H^0(S^1;\rho)$ and the 
corresponding terms cancel. See also \cite{KirkLivingston99}.
  \end{example}

\begin{example}[Lens spaces]
Let $p,q$ be integers so that $p\geq 2$, $p\geq q\geq 1$ and $p$ and $q$ are 
coprime. View the three-sphere as the unit sphere in $\CC^2$: 
$S^3=\{(z_1,z_2)\in \CC^2\mid \vert z_1\vert ^2+\vert z_2\vert^2=1\}$, 
and consider the action of 
the cyclic group of order $p$ generated by the transformation
$$
\begin{array}{rcl}
 S^3 & \to & S^3 \\
 (z_1,z_2)& \mapsto & (e^{{2\pi i}/{p}} z_1, e^{ {2\pi i \, q}/{p}} z_2)\, .
\end{array}
$$
The quotient by this action is the Lens space $L(p,q)=S^3/\!\sim\,$.
Consider the Heegaard decomposition into two solid torus $L(p,q)=V_1\cup V_2$, 
where $V_1$ and $V_2$ are the torus that lift respectively to 
$$
\{(z_1,z_2)\in S^3 \mid \vert z_1\vert \geq \vert z_2\vert\} \qquad\textrm{ and 
}\qquad \{(z_1,z_2)\in S^3 \mid \vert z_1\vert \leq \vert z_2\vert\}.
$$
Consider a non-trivial representation $\rho\colon\pi_1 L(p,q)\to \CC^*$. As the 
group is cyclic, every irreducible complex representation is one dimensional.
The representation  is determined by a non-trivial $p$-th root of unity  
$\zeta=e^{\frac{2\pi i\, k}p}$, $k\not\equiv 0\mod p$. If $\rho$ maps the soul 
of $V_1$ to $\zeta$, then it maps
the soul of $V_2$ to $\zeta^r$, where $r$ is an integer satisfying $q r\cong 
1\mod p$. Applying Mayer-Vietoris (Proposition~\ref{prop:MV}) to the pair 
$(V_1,V_2)$ and 
by the previous examples:
$$
\tor(L(p,q),\rho)=\frac{1}{  (1-\zeta)(1-\zeta^r)}.
$$
Notice that the determinant of $\rho$ is not one, but has norm one. Thus
the topological invariant is obtained once taking the module  
$ \vert \tor(L(p,q),\rho)\vert$ and considering all nontrivial $p$-roots
of unity.
This is the original example of Franz \cite{Franz35} and Reidemeister 
\cite{Reidemeister35}.  See also  \cite{Milnor66,RMK67, Cohen73}, for instance.
\end{example}

 \begin{example}[$\Sigma\times S^1$]

 \label{ex:surfacetimes}
 
Consider the product of a compact oriented surface, possibly with boundary, and 
the circle, $\Sigma\times S^1$. Let $\rho\colon\pi_1 ( \Sigma\times S^1)\to 
\SL_n(\CC)$
be an irreducible representation. In particular, $\rho$ maps $t$ the generator 
of the factor $\pi_1 S^1$ to a central matrix, namely
$\rho(t)= \omega \operatorname{Id}$ with $\omega^n=1$.
We view  $\Sigma\times S^1$ as the mapping torus of the identity on  $\Sigma$ 
and apply Proposition~\ref{prop:MapTorus}.
By irreducibility $H_0(\Sigma;\rho)\cong H_2(\Sigma;\rho)=0$ and  $\dim 
H_1(\Sigma;\rho)= -n \,\chi(\Sigma)$.
As the action of $t$ on $H_1(\Sigma;\rho)= \CC^{-n \,\chi(\Sigma) }$ is 
multiplication 
by $\omega$:
\begin{enumerate}[(a)]
 \item  $\rho$ is acyclic iff $\omega\neq 1$.
 \item when $\omega\neq 1$, then 
 $
  \tor(  \Sigma\times S^1,\rho)=(\omega-1)^{-n\,\chi(\Sigma)}
 $.
 \end{enumerate}
\end{example}

\begin{example}[Seifert fibered manifolds]
\label{ex:Seifert}
In \cite{Kitano94,Kitano96b} Kitano computed the Reidemeister torsion of a 
Seifert fibered three-manifold with a representation 
in $\SL_n(\CC)$. His result is reproduced here. Previously, Freed had computed a 
torsion for Brieskorn spheres \cite{Freed92}, and of course 
Franz \cite{Franz35} and  Reidemeister \cite{Reidemeister35} for lens spaces.

Let $M^3$ be a Seifert fibered manifold, whose base is a compact surface 
$\Sigma$, possibly with boundary, with $c$ cone points corresponding to singular 
fibres.
Let $\rho\colon\pi_1 M^3 \to \SL_n(\CC)$ be an irreducible representation. In 
particular $\rho$ maps the fibre  to $\omega \operatorname{Id}$,
where $\omega^n= 1$. For the $i$-th cone point, let $(\alpha_i,\beta_i)$ denote 
the Seifert coefficients, and let $c_i$ denote the loop 
such that $c_i^{\alpha_i} f ^{\beta_i}=1$. Let 
$\{\lambda_{i,1},\ldots,\lambda_{i,n}\}$ denote the eigenvalues of $\rho(c_i)$. 
Choose integers
$(r_i,s_i)$ such that $\alpha_i s_i-\beta_is_i=1$. 
Let $\dot\Sigma$ denote the surface $\Sigma$ minus the cone points, i.e.\ the 
space of regular fibres, 
so that $\chi(\dot\Sigma)=\chi(\Sigma)-c$.

\medskip

Assume  $M^3$ is not a solid nor a thick torus, then \cite{Kitano94,Kitano96b}:
\begin{enumerate}[(a)]
 \item  $\rho$ is acyclic if and only if $\omega\neq 1$ and 
$\lambda_{i,j}^{r_i}\omega^{s_i}\neq 1$, for every $i=1,\ldots,c$, 
$j=1,\ldots,n$.
 \item If it $\rho$ is acyclic, then
$$
 \tor(M^3,\rho)=(\omega-1)^{-n\chi(\dot\Sigma)}\prod_{i=1}^c\prod_{j=1}^n\frac 
1{(\lambda_{i,j}^{r_i}\omega^{s_i}-1)}.
 $$
 \end{enumerate}
 
 \medskip
 
 Not being a solid torus $S^1\times D^2$ nor a thick torus $S^1\times S^1\times 
[0,1]$ implies that $\chi(\dot\Sigma)<0$.

When the base $\Sigma$ is orientable, this is a straightforward consequence of 
Examples~\ref{ex:surfacetimes}, \ref{ex:circle}, and \ref{ex:2torus}.

When $\Sigma$ is not orientable,   choose a reversing orientation curve 
$\sigma\subset\dot\Sigma$, so that a tubular neighborhood of $\Sigma$ is 
a M\"obius band and its complement is orientable, with the same Euler 
characteristic. By using the fibration, this decomposes $M^3$ as a Seifert 
manifold
$N$ with orientable base and $Q$, the orientable circle bundle over the 
M\"obius strip, with $N\cap Q=\partial Q\cong S^1\times S^1$. 
 In particular $\pi_1N$ surjects onto $\pi_1M^3$, so the induced representation 
on $N$ is irreducible and we use the theorem in the case with orientable base.
 The curve $\sigma$ may be chosen so that $\rho\vert_{\pi_1 Q}$ is non-trivial, 
hence $Q$ and $\partial Q$ are acyclic, and the assertion about acyclicity 
follows from Mayer-Vietoris 
 (Proposition~\ref{prop:MV}) and
 the case with orientable base. In addition, in the acyclic case, the torsion 
of the 2-torus is trivial (Example~\ref{ex:2torus}), and so is the torsion of 
$Q$ that retracts to a Klein bottle 
 (hence it is also a mapping torus). Thus $\tor(M^3,\rho)=\tor(N,\rho)$ and the 
proof is concluded.

\begin{remark}
For a Seifert fibered manifold,  the torsion is constant on the components of 
the variety of representations.
\end{remark}

See \cite{Yamaguchi12b,Yamaguchi14} for the asymptotic behavior of these 
torsions.
 
\end{example}

\begin{example}[Torus knots]
The previous example may be applied to a torus knot, with coefficients $p,q$, 
that are relatively prime positive integers. It is Seifert fibered, 
with base a disc and two singular fibres, with coefficients $(p,1)$ and 
$(q,1)$.
Its components of the variety of representations in $\SL_2(\CC)$ are determined 
by the choice of the eigenvalues of the corresponding elements,
$\left\{e^{\pm \pi i\frac{k_1}{p}}\right\}$  and $\left\{e^{\pm \pi 
i\frac{k_2}{q} }\right\}$, with $0<k_1 < p $ and $0<k_2< q$ and $k_1\equiv k_2 
\mod 2$.
As the fibre is mapped to $(-1)^{k_1}$ times the identity, the representation is 
acyclic only for $k_i$ odd. The torsion is
$$
\frac{1}{\left(1-\cos\frac{\pi k_1}{p}\right)\left(1-\cos\frac{\pi 
k_2}{q}\right)}.
$$

The description of the components of the variety of representations in 
$\SL_3(\CC)$ is more involved \cite{MP15}, as there are much more possibilities 
for the eigenvalues.
For instance, for the trefoil knot ($p=2$, $q=3$) there are  no acyclic 
irreducible representations in $\SL_3(\CC)$.

For $p=2$, $q=5$, there are two components of acyclic representations in 
$\SL_3(\CC)$ \cite{MP15}.
For one of the components its eigenvalues are $\{e^{\frac{2\pi 
i}{3}},e^{\frac{2\pi i}{3}},e^{\frac{4\pi i}{3}}\}$ and 
$\{ e^{\frac{2\pi i}{15}}, e^{\frac{8\pi i}{15}}, e^{\frac{20\pi i}{15}}\}$,
for the other, the complex conjugates. For both the torsion is
$$
\frac{4/3}{1+2 \cos\frac{\pi}{5}}= 2-\frac{2}{3}\sqrt{5}.
$$

See \cite{Yamaguchi13} for a discussion on the torsion of torus 
knots.
\end{example}

 \begin{example}[Volume form on representations of surfaces]
 \label{ex:Witten}
   Let $\Sigma$ denote a surface of genus $g> 1$ and $G$ a compact Lie group.
 Denote by $X(\Sigma,G)$ the variety of representations of $\pi_1\Sigma$ in $G$ 
up to 
conjugacy. 
 Let $X^*(\Sigma,G)$ denote the subset of representations such that the 
infinitesimal 
commutator of the image is trivial.
 This is equivalent to $H^0(\Sigma; \operatorname{Ad}\rho)=0$, and
 for $G$ linear, this holds true when $\rho$ is irreducible.
 
If $\chi_\rho\in X^*(\Sigma,G)$,  by definition $H^0(\Sigma ; 
\operatorname{Ad}\rho)=0$. 
Hence $H^2(\Sigma ; \operatorname{Ad}\rho)=0$
and $H^1(\Sigma ;  \operatorname{Ad}\rho)$ is identified to  the tangent space 
of $X^*(\Sigma ; G)$  at the character of $\rho$
\cite{Goldman1984}, see Theorem~\ref{thm:Weil}. 
Reidemeister torsion on homology defines a volume form, by the formula of change 
of basis in homology \eqref{eqn:change}. 
Namely, if $2r=\dim X^*(\Sigma,G)$,
 $$
 \begin{array}{rcl}
  \vol_{\tor} \colon \bigwedge^{2r} H^1(\Sigma ;  \operatorname{Ad}\rho) & \to 
& \RR 
\\
   h_1\wedge\cdots\wedge h_{2r} & \mapsto & ( {\tor(\Sigma ,\rho, 
\{h_1,\ldots,h_{2r}\})} )^{-1}
 \end{array}
 $$
is a well defined linear isomorphism. The sign can be controlled by means of 
orientation.

The space $X^*(\Sigma ,G)$ has a well defined symplectic structure, due to 
Atiyah-Bott 
and Goldman \cite{Goldman1984}, that we denote by $\omega$.
In particular $\frac{\omega^{r}}{r !}$ is a natural volume form. Witten proved 
in   \cite{Witten91} that they are the same form:

\begin{thm}[Witten] For a compact Lie group $G$
 $$
 \frac{1}{(2\pi)^{2r}}  \vol_{\tor}  = \frac{\omega^{r}}{r !}.
 $$
\end{thm}

The proof uses a symplectic structure on a chain complex, which has been further 
developed by S\"ozen, cf~\cite{Sozen12a,Sozen12b}.
  
 \end{example}

\begin{example}
\label{ex:Johnson}
  
Johnson  used the point of view of volume for constructing the torsion from  a 
Heegaard splitting,
in a hand written paper 
that unfortunately 
was never published. 
Consider a closed 3-manifold $M^3$ with a Heegaard decomposition: i.e.\ 
$M^3=B_1\cup_\Sigma  B_2$, 
where $B_1$ and $B_2$ are handlebodies such that
$\Sigma= B_1\cap B_1=\partial B_1=\partial B_2$ is a surface of genus $\geq 2$.
This yields a commutative diagram of fundamental groups and their representation 
spaces:
\[
\xymatrix{
\pi_1 \Sigma
\ar[d]_{i_2*}
\ar[r]^{i_1*} & \pi_1 B_1 \ar[d]^{j_1*} \\ 
\pi_1 B_2 \ar[r]_{j_2*}  & \pi_1 M
}
\qquad\qquad\qquad
\xymatrix{
X^*(\Sigma,G)
 & X^*(B_1,G) \ar[l]_{i_1*}  \\ 
X^*(B_2,G)  \ar[u]^{i_2*}  
& X^*( M^3,G) \ar[l]^{j_2*} \ar[u]_{j_1*}
}
\]
for $G=\SU$.
Assume we have a representation $\rho\in X^*(M^3,G)$ that is infinitesimally 
rigid (i.e.\ $H^1(M^3 ; Ad\rho)=0$.) Then,
by using a Mayer-Vietoris argument, Johnson shows that $X^*(B_1,G)$ and 
$X^*(B_2,G)$ intersect transversally at the character
$\rho_0$ in $X^*(\Sigma,G)$. Johnson uses Reidemeister torsions to define 
volume 
forms $\operatorname{vol}_{B_i}$ and  $X^*(B_i,G)$
and $\operatorname{vol}_{\Sigma}$ on  $X^*(\Sigma,G)$ (as the  groups involved 
are free 
or surface groups). He defines
an invariant as the ratio between  $\operatorname{vol}_{\Sigma}$ and  
$\operatorname{vol}_{B_1}\wedge \operatorname{vol}_{B_2}$.
What he proves is:
$$
\operatorname{tor}(M^3, \rho_0)=  \frac{\operatorname{vol}_{B_1}\wedge 
\operatorname{vol}_{B_2}}{ \operatorname{vol}_{\Sigma}}\, ,
$$
cf.~Proposition~\ref{prop:MV}.  
Under Johnson hypothesis there are finitely many acyclic conjugacy classes of
representations in $\operatorname{SU}(2)$ and he  considers the addition of all 
Reidemeister torsions.
Using this Heegaard splitting, this is analogous to Casson's invariant
\cite{GuillouMarin92}, by taking into account additionally this volume on the 
varieties of characters. 
\end{example}

The point of view of volume \`a la Johnson has also been used by Dubois in 
\cite{Dubois05,Dubois06}, we will comment on it in 
Example~\ref{ex:dubois}.

\subsection{Analytic torsion}
\label{sec:analytic}

Consider now a smooth compact manifold $M$ and a representation 
$$\rho\colon\pi_1M\to \SL_n(\RR).$$ 
The manifold $M$ has a unique $\mathcal C^1$-triangulation, so one can view it 
as a CW-complex and compute its torsion.
Consider the associated flat bundle
\begin{equation}
 \label{eqn:flabundle}
E_\rho=\widetilde M\times \RR^n/\pi_1M,
 \end{equation}
where $\pi_1M$ acts on the universal covering $\widetilde M$ by deck 
transformations and on $\RR^n$ via $\rho$.
The space of $E_\rho$-valuated differential forms is denoted by  
$\Omega^p(E_\rho)$,  and its de Rham cohomology by 
$H^*(M;E_\rho)$. 
 To simplify, we assume that $\rho$ is acyclic, 
namely, 
by de Rham theorem we assume that  
\begin{equation}
 \label{eqn:deRham}
H^*(M;E_\rho)\cong H^*(M;\rho)=0.
 \end{equation}
We choose a Riemannian metric $g$ on $M$ and a metric $\mu$ on the bundle 
$E_\rho$ (notice that since we do not assume $\rho$ to be orthogonal, 
perhaps the metric $\mu$ cannot be chosen to be flat). This yields a metric on 
$\Omega^p(E_\rho)$. Using it, 
we may define the adjoint to the differential and the Laplacian $\Delta_p(\rho)$ 
on $\Omega^p(E_\rho)$.
As it is an elliptic operator, it has a discrete spectrum $0 < 
\lambda_0\leq\lambda_1\leq\cdots\to\infty$. 
The zeta function is defined on the complex half-plane $\operatorname{Re}(s)\geq 
n/2$:
$$
\zeta_p(s)=\sum_{\lambda_i}\lambda_i^{-s} 
$$
and it extends to the a meromorphic function on the complex plane, homomorphic 
at $s=0$.

 \begin{definition}
The analytic torsion is defined as:
\begin{equation}
\tor_{an}(M,\rho,g,h)=\exp\Big( \frac12 \sum_{p=0}^{\dim M} (-1)^p \,p 
\zeta_p'(0)  \Big).
\end{equation} 
\end{definition}

In \cite{Muller93} M\" uller proves that if $\dim M$ is odd  and, as we assume 
$\rho$ is acyclic, then it is independent of $g$ and $h$.
It can also be defined using the  trace of the heat operator:
\begin{multline}
\label{eqn:traceheatkernel}
\tor_{an}(M,\rho)  =  \\ \exp\left(  \frac12 \sum_{p=0}^{\dim M} (-1)^p p 
\frac{d\phantom{s}}{d s} 
\left.\left( 
\frac{1}{\Gamma(s)}
\int_0^\infty 
\left(  \operatorname{Tr}( e^{-t \Delta_p(\rho)} )  t^{s-1}  d\, t \right)
\right) \right|_{s=0}  \right)\, .
\end{multline}

In the following theorem  notice that we use the convention for Reidemeister 
torsion
opposite to the usual in analytic torsion.

\begin{thm}[\cite{Cheeger79,Muller93}]
Let $M$ be a closed hyperbolic manifold of odd dimension and 
$\rho\colon\pi_1(M)\to \SL_n(\mathbb{R})$. Then
$$
\tor_{an}(M,\rho)=\frac{1}{\vert \tor(M,\rho)\vert}.
$$
\end{thm}

This theorem was first proved by Cheeger \cite{Cheeger79} and M\"uller 
\cite{Muller78} independently for orthogonal representations, and later by 
M\"uller \cite{Muller93} for unimodular ones. In addition, aciclycity of $\rho$ 
is not required by choosing orthonormal harmonic basis in cohomology.

This subsection concludes recalling a theorem of Fried.
Let $\rho\colon\pi_1 M\to \operatorname{SO}(n)$ be a  representation of a 
hyperbolic manifold. For $s\in\CC$
with $\operatorname{Re}(s)$ sufficiently large, consider
$$
R_\rho(s)=\prod_\gamma\det(\operatorname{Id}-\rho(\gamma) e^{-s \, l(\gamma)})
$$
where the product runs over the prime, closed geodesics of $M$ and $l(\gamma)$ 
denotes the length of $\gamma$.
This is called the Ruelle zeta function.
 
 \begin{thm}[\cite{Fried86}]
 \label{thm:Fried}
 Let $M$ be a closed hyperbolic manifold of odd dimension and assume that 
$\rho\colon\pi_1 M\to \operatorname{SO}(n)$ is acyclic.
 Then $R(s)$ extends meromorphically to $\CC$ and 
 $$
  R_\rho(0)=\tor_{an}(M,\rho)^2.
 $$
 \end{thm}

 This theorem has been extended by Wotzke \cite{Wot08} to other representations 
of hyperbolic manifolds, see also \cite{Muller12}.

\section{Torsion of hyperbolic three-manifold with representations  in 
$\SL_2(\CC)$}
\label{sec:torsionSL2}

This section is devoted to the torsion of orientable hyperbolic 3-manifolds, 
using the representations in  $\SL_2(\CC)$
obtained as a lift of the holonomy representation (the choice of the lifts 
depends on  the choice of a spin structure).
It starts with closed manifolds, for which this representation is acyclic.
Then it considers  manifolds of finite volume with one cusp, for which it is 
also acyclic. Besides the invariant itself, 
it analyzes the function on the variety of characters defined by the torsion. 
This is applied for instance 
to study the behavior of torsion under Dehn filling.

\subsection{Torsions from lifts of the holonomy representation}

Let $M^3$ be a closed, orientable hyperbolic 3-manifold. Its holonomy 
representations is unique up to conjugation:
$$
\operatorname{hol}\colon \pi_1 M^3\to \operatorname{Isom}^+(\mathbb H^3)\cong 
\operatorname{PSL}_2(\CC).
$$
To get a natural representation in a linear group one can lift the holonomy to  
$\SL_2(\CC)$.
Such a lift always exists \cite{Culler86} and it depends naturally on  a choice 
of a spin structure, because
the group of isometries is naturally identified with the frame bundle of 
$\mathbb H^3$ and 
$\SL_2(\CC)$ with the spin  bundle, cf.~\cite{MFP14}.

\subsubsection{Lifts of the holonomy}

There is natural action of  $ H^1(M^3 ; \mathbb{Z}/2\mathbb{Z})$ on the set of 
lifts $\varrho$ to $\SL(2,\CC)$ of the holonomy representation: 
 viewing  $ H^1(M^3 ; \mathbb{Z}/2\mathbb{Z})$ as morphisms from $\pi_1M^3$ to 
$\mathbb{Z}/2\mathbb{Z}$,
a morphism $\epsilon\colon\pi_1M^3\to \mathbb{Z}/2\mathbb{Z}$ maps a 
representation $\varrho$ to $(-1)^\epsilon \varrho$.

\begin{prop}
\label{prop:spin}
\begin{enumerate}[(a)]
 \item \label{item:affine} There is a natural bijection between the set of lifts 
of the holonomy representation and the set of spin structures. 

 This is  an isomorphism of affine spaces on the vector space $ 
H^1(M^3 ; \mathbb{Z}/2\mathbb{Z})$.
\item \label{item:cplxconj} If $M^3$ and $\overline M^3$ are the same manifold 
with opposite orientations, then there is a natural bijection between spin 
structures 
on $M^3$ and on $\overline M^3$ so that lifts of the holonomy correspond to 
complex conjugates.
\end{enumerate}
 \end{prop}
 
Item~\eqref{item:affine} can be found essentially in \cite{Culler86}, see also 
\cite{MFP14}.

A quick way of proving  Item~\eqref{item:cplxconj} is  using the bijection of 
Item~\eqref{item:affine}, knowing that complex conjugation in
$\PSL_2(\CC)$ is the result of composing with an isometry that changes the 
orientation of $\mathbb H^3$. On the other hand, 
this bijection can be constructed  explicitly
from frame bundles and spin, but details are not given here.

\medskip

For a spin structure $\sigma$, the corresponding lift of the holonomy, according 
to 
Proposition~\ref{prop:spin}\eqref{item:affine}, will be denoted by 
 $$
 \varrho_{\sigma}\colon \pi_1 M^3\to \SL_2(\CC).
 $$

The behavior of torsion by mutation is also interesting. Mutation is the 
operation that consists in cutting along a genus two surface, applying the 
hyperelliptic involution and gluing again.
Notice that one does do not require the surface to be essential, thus the genus 
two surface can be replaced by a (properly embedded) 
torus with one or two punctures or a sphere with three or four punctures (in 
particular a Conway sphere for a knot exterior).
See~\cite{DGST12}.

Let $(M^3)^\mu$ denote the result of mutation, by \cite{Ruberman87} $(M^3)^\mu$ 
is hyperbolic with the same volume as $M^3$.

\begin{remark}
\label{rem:mutation}
 There is a natural correspondence between the spin structures on $M^3$ and the 
spin structures on $(M^3)^\mu$.
\end{remark}

Here is an explanation of the remark, using the  natural bijection between lifts 
of the holonomy and spin structures in Proposition~\ref{prop:spin}. 
Assume that $\Sigma$ separates $M^3$ in two components $M_1$ and $M_2$. Then 
$$
\pi_1M^3\cong \pi_1M_1 *_{\pi_1\Sigma}\pi_1 M_2.
$$
If $\varrho_\sigma\colon\pi_1(M^3)\to \SL_2(\CC)$ is a lift of the holonomy, 
then $\varrho_{\sigma(\mu)}\colon\pi_1 (M^3)^\mu\to \SL_2(\CC)$ is defined so 
that
$  \varrho_{\sigma(\mu)}   \vert_{\pi_1 M_1}=\varrho_\sigma\vert_{\pi_1 M_1}$ 
and $\varrho_{\sigma(\mu)}\vert_{\pi_1 M_2}$ is conjugate to 
$\varrho_\sigma\vert_{\pi_1 M_2}$ by a matrix in $\PSL_2(\CC)$
that realizes the involution $\mu$ on $\Sigma$.
When $\Sigma$ does not separate, then the construction is similar from the 
presentation of $\pi_1M^3$ as HNN-extension.

 \subsubsection{Torsions for closed 3-manifolds}
 
 The following theorem is a particular case of Raghunathan's. With other 
cohomology results, it is discussed in 
 Appendix~\ref{sec:cohomology}. In particular the following theorem is stated in 
Corollary~\ref{cor:acyclc}.

\begin{thm}
\label{thm:RaghunathanClosed}
Let $M^3$ be a closed, orientable, hyperbolic 3-manifold. Then any lift of its holonomy representation is acyclic.
 \end{thm}

 With Mostow rigidity, this yields immediately a topological invariant of the spin manifold.

 \begin{definition}
 \label{defn:tor2closed}
  Let $M^3$ be a compact, oriented, hyperbolic 3-manifold with spin structure $\sigma$. The torsion of $(M^3,\sigma)$ is defined as
  \begin{equation}
 \label{eqn:tau}
 \tau(M^3,\sigma):=\tor(M^3, \varrho_\sigma)\in \mathbb{C}^*
 \end{equation}
 where $\varrho_\sigma=\widetilde{\mathrm{hol}}$ is the lift of the holonomy $\mathrm{hol}$ corresponding to the spin structure $\sigma$.
 \end{definition}

 \begin{remark}
  There is no sign indeterminacy, i.e.\  it is a well defined complex number, 
because $  \varrho_\sigma $ is a representation in $\mathbb C^2$, 
  which is even dimensional,
  and $\chi(M^3)=0$, see Remark~\ref{rem:sign}.
   \end{remark}

 Here are some of its properties.
 
 \begin{prop}
 \label{prop:properties}
  The torsion $ \tau(M^3,\sigma)$ in Definition~\ref{defn:tor2closed}  has the following properties:
  \begin{enumerate}[(a)]
   \item \label{item:topinvariance} It is a topological invariant of the spin manifold $(M^3,\sigma)$.
   \item  \label{item:sensitive} There are examples of manifolds $M^3$ with two spin structures $\sigma$ and $\sigma'$ such that 
   $\tau(M^3,\sigma)\neq \tau(M^3,\sigma')$.
   \item \label{item:conj} Let $\overline M^3$ denote the manifold $M^3$ with opposite orientation. If one changes the orientation and the spin structure accordingly 
   as in Proposition~\ref{prop:spin}, then the torsion is the complex conjugate
   $$
         \tau(\overline M^3,\overline \sigma)=     \overline{   \tau(M^3,\sigma)}. 
   $$
   \item \label{item:mutation} Let  $(M^3)^\mu$ denote the result of mutation. If $\sigma^\mu$ denotes the corresponding spin structure as in Remark~\ref{rem:mutation}, then 
   $$
	\tau((M^3)^\mu, \sigma^\mu) = \tau(M^3,\sigma).
   $$
   \item \label{item:dense} Let $M^3$ be an oriented hyperbolic manifold    with one cusp and with spin structure $\sigma$. 
   The set of modules of the torsions obtained by Dehn filling on $M^3$, 
   $\vert \tau(M^3_{p/q},\sigma)\vert $ so that  $\sigma$ extends to $M^3_{p/q}$, is dense in the interval 
   $$\left[ \frac14\left\vert \tau(M^3,\sigma)\right\vert,+\infty\right).
   $$
  \end{enumerate}
\end{prop}

 Item \eqref{item:topinvariance} follows from uniqueness of the hyperbolic structure, by Mostow rigidity. 
To prove \eqref{item:sensitive} it suffices to compute an example, this is done in Corollary~\ref{cor:differentspin}.
Item \eqref{item:conj} is straightforward from Proposition~\ref{prop:spin}. Item \eqref{item:mutation} is proved in \cite{MFP12c}.
Finally, \eqref{item:dense}  is proved later when discussing cusped manifolds, as this will follow immediately from a surgery formula. 
 
\medskip
 
Item \eqref{item:dense} shows that this torsion is not obviously related to the 
hyperbolic volume.
The following theorem, a particular case  of  \cite[Thm.~4.5]{BV13}, finds a 
relation (we use the convention of torsion opposite to \cite{BV13}):

 \begin{thm}[Bergeron and Venkatesh, \cite{BV13}]
 \label{thm:BV}
 Let $M^3$ be a compact oriented hyperbolic 3-manifold. Assume that $M_n^3\to M^3$ is a sequence of coverings such that the injectivity radius of $M_n^3$ converges to infinity.
 Then
 $$
\lim_{n\to\infty} 
\frac{\log\vert \tau(M_n^3,\sigma)\vert}{ \operatorname{vol}(M_n^3) } = 
\frac{11}{12 \pi} .
 $$ 
 Equivalently;
 $$
\lim_{n\to\infty} \frac{\log\vert \tau(M_n^3,\sigma)\vert} {\deg ( M_n^3\to M^3) 
} = \frac{11}{12 \pi} \operatorname{vol}(M^3) .
 $$ 
 \end{thm}

 This theorem relies on analytic torsion and on $L^2$-torsion, as $\frac{-11}{12 
\pi}$ is the $L^2$-torsion of $\mathbb H^3$. The proof uses the $L^2$-Laplacian 
of 
 hyperbolic space, and it is based on approximations of averages of  the trace of the difference of heat kernels,
 see Equation~\eqref{eqn:traceheatkernel}.
  They require the notion of  strong acyclicity (the property in Theorem~\ref{thm:strong_acyclicity}) to avoid eigenvalues of the Laplacian approaching to zero.
  This has been generalized in  \cite{ABBGNRS}, in particular without requiring that the 
 $M_n^3$ are coverings. See also \cite{MuP14}.

\subsubsection{Cusped hyperbolic 3-manifolds}
 \label{section:cusped}
 
 In this subsection, $M^3$ denotes a finite volume hyperbolic manifold, i.e.\ a manifold whose ends are cusps. 
 
\begin{assumption}
Assume that $M^3$ has a single cusp.
\end{assumption}
 
This is done not only to simplify notation, but because with more cusps some 
further issues need to be discussed \cite{MFP14}.
Again one has:
 
\begin{thm} Let $M^3$ be a closed, orientable, hyperbolic 3-manifold with one cusp. Then any lift to $\SL_2(\CC)$ of its holonomy representation is acyclic.
 \end{thm}
 
This is proved for instance in \cite{MFP12}, it is a  particular case of 
Theorem~\ref{prop:dimHM} in Appendix~\ref{sec:cohomology}. 
With more cusps this is may not hold true for all lifts of the holonomy 
representation, i.e.~for all spin structures. 
It is true provided that  for each cusp the trace of the peripheral elements is 
not identically $+2$ 
(for some elements it is $-2$). This is always the case if there is a single 
cusp \cite{MFP12, Calegari06}.

  One may as well define the same torsion as in Definition~\ref{defn:tor2closed}:

 \begin{definition}
 \label{defn:tor2cusped}
  Let $M^3$ be a compact, oriented, hyperbolic 3-manifold with one cusp, and let $\sigma$ denote a spin structure on 
  $M^3$. The torsion of $(M^3,\sigma)$ is defined as
  \begin{equation}
 \label{eqn:taucusp}
 \tau(M^3,\sigma):=\tor(M^3,  \varrho_\sigma )\in \mathbb{C}^*\, ,
 \end{equation}
 where $\varrho_\sigma=\widetilde{\mathrm{hol}}$ is the lift of the holonomy corresponding to $\sigma$.
 \end{definition}
 
This torsion has the same properties as in the closed case, Proposition~\ref{prop:properties}. 
 
 \subsubsection{The twisted polynomial}

It is relevant to mention the twisted  polynomial for hyperbolic knot exteriors corresponding to a lift of the holonomy $\varrho$
constructed by  Dunfield, Friedl, and Jackson in 
\cite{DFJ12}. 
Given a hyperbolic knot ${\mathcal K}\subset S^3$, choose the spin structure on 
$M^3=S^3-{\mathcal K}$ such that the trace of the meridian is $+2$ 
(the trace of the longitude is always -2 by \cite{Calegari06}, see also 
\cite[Corollary 3.10]{MFP12}) and consider the abelianization $\phi\colon\pi_1 
M^3\to \ZZ$. 
Dunfield, Friedl, and Jackson
study the polynomial
$$
 \Delta_{\mathcal K}(t):=\Delta_{M^3, \varrho\otimes\phi}(t)
$$
following the notation of  Subsection~\ref{subsection:polynomial}.
By Proposition~\ref{prop:degree}   its degree is $\leq 2 (2 g({\mathcal K})+1)$, where $g({\mathcal K})$ is  the genus of the knot, i.e.\ the minimal genus of a Seifert surface.
Numerical evidence (knots up to 15 crossings) yield them to conjecture:
 
 \begin{conj}[\cite{DFJ12}]
For a hyperbolic knot $\mathcal K$:
\begin{enumerate}[(a)]
 \item $\deg \Delta_{\mathcal K}(t) = 2 (2 g({\mathcal K})+1)$.
 \item $\Delta_{\mathcal K}$ is monic if and only if ${\mathcal K}$ is s fibered knot.
\end{enumerate}
 \end{conj}

The equality of the degree  has been proved  by Morifuji  and Tan for some 
families of two bridge knots, see \cite{Morifuji15} and references therein. Agol 
and Dunfield have shown:

\begin{thm}[\cite{AD15}] For libroid hyperbolic knots,  
$$
 \deg  \Delta_{\mathcal K}(t) = 2 (2 g({\mathcal K})+1).
$$ 
\end{thm}

   Being libroid means the existence of a collection of disjointly embedded 
minimal genus Seifert surfaces in the exterior of the knot
so that their open complement is a union of books of $I$-bundles, in a way 
that respects the structure of sutured manifold.
See \cite{AD15}.
   
In the remarkable paper \cite{DFJ12} the authors also conjecture that being 
monic determines whether the knot is fibered or not, and rise many interesting 
questions about this polynomial and its relationship with other invariants.

 \subsection{Torsion on the variety of characters}
Let $M^3$ be a hyperbolic, oriented manifold with one cusp. A relevant 
difference with the closed case is the fact that the holonomy  of $M^3$ can be 
deformed in the variety of representations
(to holonomies of non-complete structures).

\subsubsection{The distinguished curve of characters}
The variety of $\SL_2(\CC)$-representations of $M^3$ is the set
$$
 \hom(\pi_1M^3,\SL_2(\CC)),
$$
which it is an affine algebraic set: if  a generating set of $\pi_1M^3$ has $k$ 
elements, then  $\hom(\pi_1M^3,\SL_2(\CC))$ embeds in $\SL_2(\CC)^k\subset  
\CC^{4 k}$, by mapping a representation to the image of its generators. The 
algebraic equations are induced by the relations of the group.

The group $\PSL_2(\CC)$ acts on $
\hom(\pi_1M^3,\SL_2(\CC))
$ by conjugation, and the affine algebraic quotient is the variety of characters
$$
 X(M^3):= X(M^3,\SL_2(\CC))=  \hom(\pi_1M^3,\SL_2(\CC))/\! /  \PSL_2(\CC).
$$
This is defined in terms of the invariant functions: $X( M^3,\SL_2(\CC))$ is the algebraic affine set 
whose function ring $\CC[  X(M^3,\SL_2(\CC)) ]$ is the ring of invariant functions
$$
\CC[ \hom(\pi_1M^3,\SL_2(\CC))]^{\PSL_2(\CC)}.
$$

By \cite{ThurstonBook}, see also \cite[Appendix B]{BoileauPorti}  each 
component that contains the lift of the holonomy of $M^3$ is a curve.

\begin{definition}
 An irreducible component of $X(M^3,\SL_2(\CC))$ that contains a lift of the 
holonomy is called
 a \emph{distinguished component} and it is denoted by $X_0(M^3)$.
\end{definition}

For many manifolds, e.g.\ for 2-bridge knot exteriors, there is a unique 
distinguished component.  A priori there could be more components,
but the definition makes sense  because they would be isomorphic.
More precisely, 
there are two characters of the holonomy representation in $\PSL_2(\CC)$ that 
are complex conjugate from each other, that correspond to the different 
orientations. When lifting them to $\SL_2(\CC)$, this gives
$2 \vert   H^1(M^3;\ZZ/2\ZZ) \vert  $
characters, two for each spin structure. 
The corresponding components $X_0(M^3)$ are isomorphic by means of 
the natural action of $H^1(M^3;\ZZ/2\ZZ)$ and complex 
conjugation.

Recently, Casella, Luo, and Tillmann \cite{CLT}
have shown an example of hyperbolic manifold with one cusp $M^3$ such that 
the holonomy characters  of the different orientations lie 
in different components of $X(M^3,\PSL_2(\CC))$. To my knowledge, the following 
question is still open:

\begin{question} Once $M^3$ is \emph{oriented}, 
are all the lifts of the oriented holonomy contained in a single irreducible 
component 
of  $X(M^3,\SL_2(\CC))$?
\end{question}

The distinguished component $X(M^3,\SL_2(\CC))$  is a curve and it was studied 
by Thurston in his proof of the hyperbolic Dehn filling theorem 
\cite{ThurstonBook}. More precisely, in a neighborhood of the
holonomy of the complete structure of $M^3$, the representations are holonomies 
of incomplete structures, and in some cases the completion is a Dehn
filling. This is discussed in Paragraph~\ref{sec:Dehnfilling}.

\subsubsection{Torsion on the distinguished curve of characters}

We say that a character $\chi\in X(M^3)$ is trivial if it takes values in $\{\pm 2\}$, i.e.\ it is a lift or the trivial character in $\PSL_2(\CC)$.

An irreducible character is the character of a unique conjugacy class of representations. 
A reducible character can correspond to more conjugacy classes, but if the character is
non-trivial, then either all representations with this character are acyclic, either 
none of them is (Lemma~\ref{lemma:acyclicorbit}).

\begin{definition} Define the torsion function on $X_0(M^3)$ minus the trivial character: 
\begin{equation}
 \label{eqn:torsionfunct2}
\mathbb{T}_M(\chi_\rho)=\left\{
\begin{array}{ll}
\tor(M^3,\rho) & \textrm{if }\rho\textrm{ is acyclic;} \\
0 &  \textrm{if } \chi_\rho \textrm{ is non-trivial and } \rho \textrm{ 
non-acyclic,}
\end{array}
\right. 
\end{equation}
where $\chi_\rho$ denotes the character of $\rho$. 
\end{definition}

\begin{prop}
\label{prop:rational}
 For a hyperbolic oriented  manifold with one cusp, the torsion defines 
a rational function on $X_0(M^3)$,
 $\mathbb{T}_M\in \CC(X_0(M^3))$,
which is regular away from the trivial character.

In particular, if the trivial character is not contained
in $X_0(M^3)$ (e.g.~if $b_1(M^3)=1$), then  $\mathbb{T}_M\in \CC[X_0(M^3)]$, 
i.e.~it is holomorphic (with no poles)
$$
\mathbb{T}_M \colon X_0(M^3)\to \CC.
$$
\end{prop}

 \begin{remark}
We prove that the trivial representation cannot be approached by irreducible ones when $b_1(M^3)=1$
in Appendix~\ref{sec:notapprox}.  
This is always the case when $M^3$ is a knot exterior.
 \end{remark}

\begin{proof}
The dimension of each cohomology group is upper semi-continuous on the representation
 (see \cite[Lemma~3.2]{Heusener-Porti2011}, 
this is a particular case of the semi-continuity theorem \cite[Ch.~III, Theorem 12.8]{Hartshorne1977}).
With Lemma~\ref{lemma:acyclicorbit}, we can conclude that  
acyclicity holds true in a dense open Zariski domain $U\subset  X_0(M^3)$, after removing
the trivial character if required.
The fact that the function $\mathbb{T}_M$ is algebraic on this domain is clear,  as this is defined from polynomials
on the entries of $\rho$.
Invariance by conjugation is one of the properties of the torsion. 
Thus it remains to deal with the points where it is not acyclic.
Recall that by Lemma~\ref{lemma:acyclicorbit} a representation with nontrivial character  is acyclic
if and only all representations with the same character are.

First notice that  $H^0(M^3;\rho)$ is trivial when the character
$\chi_\rho$ is non-trivial,  because this cohomology group
is  naturally isomorphic to the space of invariants $\CC^{\rho(\pi_1 M^3)}$. More precisely,
$\CC^{\rho(\pi_1 M^3)}$
 is non-trivial only when all elements in $\rho(\pi_1 M^3)$ have $1$ as eigenvalue, which
means that their trace is $2$, i.e.\ $\chi_\rho$ is trivial.
Thus by duality $H_0(M^3;\rho)=  0$ when $\chi_\rho$ is non-trivial.

Now fix a representation $\rho_1$ which is not acyclic and non-trivial.
Non-acyclicity implies that  $H_1(M^3;\rho_1)\neq 0$ and $H_2(M^3;\rho_1)\neq 0$,
as $H_0(M^3;\rho_1) = 0$, the homotopical
dimension
of $M^3$ is $2$, and $\chi(M^3)=0$.
Notice that $M^3$ has the simple homotopy type of a 2-complex (see 
\cite[Page~54]{Nicolaescu} for instance), that can be used to compute the 
torsion.
Using the notation of Section~\ref{sec:comb},
 fix a basis $\{v_1,v_2\}$ for $\CC^2$ and  lifts of cells $\tilde e^i_j$  of a triangulation of $M^3$.
Then define a family of basis $c_i(\rho)$ by varying the representation in $v_k\otimes_\rho \tilde e^i_j$.
Now choose $\tilde b_1(\rho)=c_2(\rho)$
and $\tilde b_0(\rho)$, a  linear combination of $c_1(\rho)$ with constant coefficients (though $c_1(\rho)$ changes
with $\rho$), so that $\partial(\tilde b_0(\rho_1))=c_0(\rho_1)$. 
The function
\begin{equation}
\label{eqn:ratiotau2}
\rho\mapsto [\partial \tilde b_1(\rho)\sqcup \tilde b_0(\rho),c_1(\rho) ]/[\partial \tilde b_0( \rho),c_0(\rho)]
\end{equation}
is well defined in the set where its denominator does not vanish. This is a Zariski open set 
that contains $\rho_1$.
On this set $\partial \tilde b_1(\rho)=\partial c_2(\rho)$ has maximal rank iff $\rho$ is acyclic, thus the function 
  \eqref{eqn:ratiotau2} vanishes when $\rho$ is not acyclic, and when $\rho$ is acyclic 
   \eqref{eqn:ratiotau2} is the torsion.
\end{proof}

\begin{example}
\label{ex:fig8sl2}
 For the figure eight knot, in \cite{Kitano94b} Kitano computes it:
 $$
 \mathbb{T}_M(\rho)=2-2\tr(\rho(m))=2-2\chi_\rho(m)\, ,
 $$
 where $m$ denotes the meridian of the knot. Notice that here the function  only 
depends on $\tr(\rho(m))$,
 namely the evaluation of the character at $m$, and one does not 
 need to describe the variety of characters. 
 In general, as $ \tr(\rho(m))$
 is a non-constant function on the curve $X_0(M^3)$, $\mathbb{T}_M(\rho)$ and  $ \chi_\rho(m)= \tr(\rho(m)) $ are related by a polynomial equation;
 compare with \cite{DuboisG16}.
\end{example}

 \subsubsection{Dehn filling space}
\label{sec:Dehnfilling}

Again let $M^3$ be an oriented, finite volume hyperbolic  manifold with one cusp. 
Consider the peripheral torus $T^2$, which is  the boundary of a compact core   of $M^3$.
Choose a frame on the peripheral torus $T^2$, i.e.\ two simple closed curves 
that generate $\pi_1 T^2$, denote this frame by
$\{m,l\}$. The notation suggests that the canonical choice for a knot exterior is the pair  meridian-longitude.

A \emph{Dehn filling} on $M^3$ is the result of gluing a solid torus $D^2\times S^1$ to a compact core $ M^3$ along the boundary.
Up to homeomorphism, this manifold depends only on the (unoriented) homology class in $T^2$ of the meridian, 
i.e.\ the curve $\partial D^2\times \{*\}$. This curve is written as $\pm (p\, m+q\, l)$, and the Dehn filling is 
denoted by $M^3_{p/q}$.

 When $\vert p\vert +\vert q\vert$ is sufficiently large, by Thurston's
theorem $M_{p/q}$ is hyperbolic.   To prove it, he introduces the \emph{Thurston's parameters} of the Dehn filling space,
by writing, for representations close to the holonomy of the complete hyperbolic structure,
\begin{equation}
 \label{eqn:u}
\rho(m)=\pm \begin{pmatrix}
         e^{u/2} & 1 \\ 0 & e^{-u/2}
        \end{pmatrix}
        \qquad
 \rho(l)   =  \pm \begin{pmatrix}
         e^{v/2} & f(u) \\ 0 & e^{-v/2}
        \end{pmatrix} 
\end{equation}
with $u,v\in\CC$ in a neighborhood of the origin. 
 For the holonomy of the complete structure, $u=v=0$ and   write 
 $$
 \cs= \cs(l,m)=f(0)\in \CC-\RR. 
 $$
 
\begin{definition}
\label{def:cs}
 The parameter $u\in U \subset \CC$ as above is called the \emph{Thurston parameter} and  $\cs= \cs(l,m)\in \CC-\RR$ the \emph{cusp shape}
 of the complete structure.
 \end{definition}

The Thurston parameter $u$ is in fact a parameter of a double branched covering of a neighborhood of the lift of the holonomy $\varrho$, 
as the local parameter of $X_0(M^3)$ is
$\tr(\rho(m))= \pm 2 \cosh\frac{u}{2} $.
For the complete structure, the peripheral group acts as a lattice on a horosphere, that we view as $\CC$. Up to affine equivalence, this lattice is given by
$m\mapsto 1$, $l\mapsto \cs$.

\begin{lemma}[Neumann and Zagier \cite{NZ85}] There exists a (standard) neighborhood of the origin $U\subset \CC$ such that:
\begin{enumerate}[(a)]
\item The map $ U\to X_0(M^3)$ such that $u\mapsto \chi_\rho$, where $\rho$ is as in \eqref{eqn:u} is a double branched covering 
of a neighborhood in $X_0(M^3)$ of the character of the holonomy
of the complete structure.
 \item $v$ is an analytic odd function on $u$ that satisfies 
 $$v(u)=\cs u+ O(u^3).
 $$
 \item $f(u)=\sinh(v)/\sinh(u)$.
\end{enumerate}
\end{lemma}

The \emph{generalized Dehn filling coefficients} are the 
$
(p,q)\in\RR^2\cup\{\infty\}
$ such that 
\begin{equation}
 \label{eqn:pq}
p\,u+q\, v=2\pi i 
\end{equation}
when $u\neq 0$ and 
$\infty$ when $u=0$.

\begin{thm}
 \cite{ThurstonBook,NZ85}
 The generalized Dehn filling coefficients define a homeomorphism between $U$ and a neighborhood of $\infty$ in
 $S^2=\RR^2\cup\{\infty\}$. 
 If a  pair of coprime integers $(p,q)\in\ZZ^2$  lies in the image of this homeomorphism,
 then $M^3_{p/q}$ is hyperbolic, with holonomy whose restriction to $M^3$ satisfies \eqref{eqn:pq}. 
\end{thm}
 
We are interested in properties of $M^3_{p/q}$.

\begin{thm}[Neumann and Zagier \cite{NZ85}]
 $$\vol (M^3_{p/q})=\vol (M^3)-\pi^2\frac{\operatorname{Im}(\cs)}{|p+\cs q|^2}+ O\left(\frac{1}{|p+\cs q|^4 }\right).$$
\end{thm}

We are also interested in the complex length of the soul  of the solid torus 
added by Dehn filling, which is a short geodesic.
Let $r$ and $s$ be integers satisfying $p\, s-q\, r=1$,  this complex length is 
$r\, u+s\, v$. A straightforward computation yields
\cite{Memoirs}:

\begin{remark}
 \label{rem:length}
The complex length of the core of the solid torus $M_{p/q}^3$ is
$$
r\, u+s\, v= 2\pi i\frac{r}{p}+ \frac{v}{p}= 2\pi i\frac{s}{q}+ \frac{u}{q}
$$
\end{remark}

Given a spin structure $\sigma$ on $M^3$, it may extend or not to $M_{p,q}^3$. It is easy to give a characterization 
using the bijection of Proposition~\ref{prop:spin}. 

\begin{lemma}\cite{MFP14}
\label{lemma:extendspin}
A spin structure $\sigma$ on $M$ extends to $M_{p,q}$ iff the corresponding lift of the holonomy  $\varrho_{\sigma}$ 
(for the complete structure on $M^3$) satisfies
$$
\tr(\varrho_{\sigma}( p\, m+ q\, l ))=-2.
$$
\end{lemma}

\subsubsection{Torsion for Dehn fillings}

Let $\vert p\vert +\vert q\vert$  be  sufficiently large and let $\sigma$ be a spin structure on $M^3$ extensible to $M_{p/q}^3$. 
Let $\rho_{p/q,\sigma}$ denote the lift of the holonomy of the  restriction to 
$M^3$ of the hyperbolic structure on  $M_{p/q}^3$, and $\chi_{p/q,\sigma}$ the corresponding character.
Let $\lambda_{p/q}$ denote the complex length of the soul of the filling torus.
There is a surgery formula:

\begin{prop}
\label{prop:SF}
Let $M^3$ be as above. For $\vert p\vert +\vert q\vert$ sufficiently large:
 $$
 \tau(M_{p/q}^3,\sigma)=  \frac{\mathbb{T}_{M}(\chi_{p/q,\sigma} )}{ 2 
(1-\cosh(\lambda_{p/q}/2 ) ) }.
 $$
\end{prop}

This proposition can be proved using Mayer-Vietoris (Proposition~\ref{prop:MV}) 
to the pair 
$(M^3, D^2\times S^1 )$. The proof can be found for instance in 
\cite{Kitano15} 
for the figure eight knot,
and in \cite{Tran15} for twist knots, 
but it applies to every cusped manifold, and it is essentially done too in  \cite[Proposition~4.10]{Memoirs}. The only computation required is the torsion of the
solid torus added by Dehn filling, or its core geodesic, which is (by Example~\ref{ex:circle}):
$$
\frac{1}{
\det \left( \begin{pmatrix}
            1 & 0 \\
            0 & 1
           \end{pmatrix}
           -
           \begin{pmatrix}
            e^{\lambda_{p/q}/2} & 0 \\
            0 & e^{-\lambda_{p/q}/2}
           \end{pmatrix}
      \right) }= \frac1{2 (1-\cosh(\lambda_{p/q}/2 ))}.
$$

\begin{remark}
\label{rem:2pi}
 The complex length $\lambda_{p/q}$ is defined up to addition of a term in $2\pi i \ZZ$. The spin structure (or equivalently a choice of lift of the holonomy)
 determines $\lambda_{p/q}$ up to $4\pi i\ZZ$, as the holonomy of this curve is conjugate to
 $$
 \begin{pmatrix}
 e^{\lambda_{p/q}/2} & 0 \\ 0 & e^{-\lambda_{p/q}/2}
\end{pmatrix}.
 $$
 \end{remark}

\begin{cor}
\label{cor:differentspin}
 There exist closed hyperbolic manifolds $N^3$ with two spin structures $\sigma$ and $\sigma'$ so that 
 $\tau(N^3,\sigma)\neq \tau(N^3,\sigma')$.
\end{cor}

\begin{proof}
The manifolds are obtained by Dehn filling on the figure eight knot. Consider 
the sequence $p_n/ q_n= 2 n$ so that both spin structures
on the figure eight knot exterior $M^3$ extend to $M^3_{p_n/q_n}$, i.e.\ 
$H_1(M^3_{2 n} ; \ZZ/2\ZZ)\cong \ZZ/2\ZZ$.
The corresponding  lifts differ on the sign of the trace of the meridian.
As the core of the solid torus is homologous to the meridian, the
sign of the traces of these cores differ.
Hence  for one of the lifts the complex length (modulo $4\pi i\ZZ$) is $\lambda_{2n}$, for the other one it is
$\lambda_{2n}+ 2\pi i$
by Remark~\ref{rem:2pi}. By Remark~\ref{rem:length}, one may approximate 
\begin{equation}
 \label{eqn:approx}
 \lambda_{2n}=\frac{2\pi i}{2n}+ O\left(\frac{1}{n^2}\right).
\end{equation}
The goal is  to show that the limit
\begin{equation}
 \label{eqn:limit}
 \lim_{n\to+\infty} \frac{\vert \tau(M_{2 n},\sigma )\vert }{ \vert \tau(M_{2 n},\sigma' )\vert}
 \end{equation}
is not $1$. By
Proposition~\ref{prop:SF} it is the product of the limits:
\begin{equation}
 \label{eqn:limitlambdas}
 \lim_{n\to\infty}\frac{1-\cosh( \lambda_{2n}/2 ) }{ 1-\cosh( (\lambda_{2n} + 2\pi i)/2) }
=\lim_{n\to\infty}\frac{1-\cos\frac{\pi} {2n}}{1+\cos\frac{\pi} {2n}} = 0
\end{equation}
(using the approximation \eqref{eqn:approx}) and
\begin{equation}
 \label{eqn:limitfunctions}
 \lim_{n\to\infty}  
 \frac{\mathbb{T}_M(\chi_{2n,\sigma} 
)}{\mathbb{T}_M(\chi_{2n,\sigma'} )}
=
\frac{\tau(M^3,\sigma)}{\tau(M^3,\sigma')}.
\end{equation}
As for one of the spin structures the trace of the meridian is $2$ and for the other $-2$, by Example~\ref{ex:fig8sl2},
the limit in \eqref{eqn:limitfunctions} is 
  $
  (2-2\cdot(-2))/(2-2\cdot2)=-3
 $.
 Thus the limit \eqref{eqn:limit} vanishes, hence it is not $1$.
\end{proof}

The previous argument applies to any knot exterior, the precise  value of the torsion is not needed. 
Corollary~\ref{cor:differentspin} holds true also for  cusped manifolds,
by applying an analogous formula for partial surgery.

\begin{cor}
\label{cor:dense}
The set of modules of the torsions $\left\vert \tau(M^3_{p/q},\sigma)\right\vert $ obtained by Dehn filling on $M^3$,
so that the spin structure $\sigma$ on $M^3$ extends to  $M^3_{p/q}$, is dense in $$
\left[\frac14{\vert \tau(M^3,\sigma)\vert},+\infty\right).
$$
\end{cor}

\begin{proof}
By the surgery formula, Proposition~\ref{prop:SF}, it suffices to show that 
$$
  \left\vert 1- \cosh( 
  {\lambda_{p/q}}/2)\right\vert
$$ 
is dense in 
the interval $[0,2]$.
Let $r, s\in\ZZ$ be such that $p s-q r = 1$, by Remark~\ref{rem:length},  $\lambda_{p/q}= 2\pi i \frac{r}p+\frac up$ and $u\to 0$. 
Then the result follows from the density of $\cos(\pi \frac{r}{p})$ in $[-1,1]$. 
Notice that extendibility of $\sigma$ is just determined by a condition
$a\, p + b\, q\equiv 0 \mod 2$,
for   $a, b\in \ZZ/2\ZZ$ depending on $\sigma$. 
\end{proof}

As the surgery formula shows, this density is a consequence of the contribution of the core geodesics. 
Notice that in Theorem~\ref{thm:BV} the hypothesis require the length of geodesics to be bounded below away from zero.

 \subsubsection{Branched coverings on the figure eight}
\label{sssection:branched}

We next consider another family of examples. We consider it only for the figure eight knot, but it generalizes to other manifolds.

Consider $M_ n^3$ the n-th cyclic branched covering of $S^3$, branched along the figure eight knot. It is hyperbolic for $n\geq 4$.
 Its torsion may be computed by means of the twisted Alexander polynomial, by using the formula \`a la Fox, 
 Proposition~\ref{prop:Fox},
 due   to Turaev
 \cite{Turaev86},
 see also \cite{DY12,Porti04,SW02,Raimbault12, Yamaguchi12}.
 More precisely, this polynomial is
 $$
 p_{\chi}( t )=t^2-2\chi(m) t + 1,
 $$
 by \eqref{eqn:tor82t},
 cf.~\cite{Kitano-Morifuji2012}.
 Notice that its value at $1$ agrees with Example~\ref{ex:fig8sl2}.
The holonomy of $M_ n^3$ extends to a holonomy of the quotient orbifold $M_ 
n^3/ 
(\ZZ/n\ZZ)$ in $\PSL_2(\CC)$. 
It induces a representation of the exterior of the 
not $S^3-\mathcal K$ that lifts to $\SL_2(\CC)$, that we denote by $\rho_n$.
Since a rotation of angle $2\pi /n$ has trace $\pm 2\cos ( \pi/n)$, we have
 $$
 \chi_{\rho_n}(m)= \tr(\rho_n(m))= \pm 2 \cos ( \pi/n).
 $$
 Using  Proposition~\ref{prop:Fox}, if $\Sigma_n\subset M_n$ is the lift of the branching locus, then
\[
   \tor(M_n^3-\Sigma_n,\sigma)=\prod_{k=0}^{n-1} p_{\chi_n}( e^{2\pi i\frac{k}{n}})  =
 \prod_{k=0}^{n-1} \left( e^{4\pi i\frac{k}{n}}\mp 4\cos({\pi}/{n} ) 
e^{2\pi i\frac{k}{n}}+1\right).
 \]
By the surgery formula (Proposition~\ref{prop:SF}) and knowing that the complex length of the core geodesic is  
 $ \approx  \frac{\sqrt{3}\,\pi}n + O\left(\frac{1}{n^3}\right)  $ (Remark~\ref{rem:length}), 
 \begin{equation}
  \label{eq:foxsurgery}
 \tor(M_n^3,\sigma)= \frac{1/2}{1-\cosh \left( \frac{\sqrt{3}\,\pi}n  + 
O\left(\frac{1}{n^3}\right) \right)} 
  \prod_{k=0}^{n-1} \left( e^{4\pi i\frac{k}{n}}\mp 4\cos( {\pi}/{n} ) 
e^{2\pi i\frac{k}{n}}+1\right).
 \end{equation}
 In particular,
 \begin{equation}
  \label{eqn:limlog/n}
\lim_{n\to\infty} \frac{\log\vert \tor(M_n^3,\sigma)\vert}{n}= 
 \lim_{n\to\infty} \frac{1}{n} \sum_{k=0}^{n-1} \log \left\vert  e^{4\pi 
i\frac{k}{n}}\pm 4\cos( {\pi}/{n} ) e^{2\pi i\frac{k}{n}}+1 \right\vert.  
 \end{equation}
Here we used that $\lim_{x\to 0^+}  x\log x =0$ to get rid of the first term, outside the product, in \eqref{eq:foxsurgery}.
As $t^2 \mp 4 t + 1$ has no roots in the unit circle,  the limit 
\eqref{eqn:limlog/n} converges to
$$
 \frac{1}{2\pi}\int_{\vert z\vert =1} \log\vert z^2\pm  4 z +1 \vert .
$$
 Thus by Jensen's formula
 $$
 \lim_{n\to\infty} \frac{\log\vert \tor(M_n^3,\sigma)\vert}{n}
 = \log\vert 2+\sqrt{3}\vert .
 $$

 This approach uses the ideas on Mahler measure of \cite{SW02}. It applies to a hyperbolic knot provided that the torsion polynomial has no roots in the unit circle.
 It is easy to prove that the roots of the polynomial are not roots of the unity, but this does not discard roots in the unit circle. For 
 the adjoint representation, the nonexistence of roots in the unit circle has been proved by Kapovich \cite{Kapovich98}.

 \section{Torsions for higher dimensional representations of hyperbolic three-manifolds}
\label{sec:repsandtors}
 
This section is devoted to further Reidemeister  torsions naturally associated to hyperbolic three-manifolds,
those obtained by using the (finite dimensional and linear) representations of $\SL_2(\CC)$.

 \subsection{Representations of $\SL_2(\CC)$}
 \label{section:reps}

Let $V_{k+1}$ be the space of degree $k$ homogeneous polynomials in two variables:
 $$
 V_{k+1}:=\{ P\in \mathbb{C}[X,Y]\mid P \textrm{ is homogeneous of degree } k\}.
 $$
 The group
$\SL_2(\CC)$ acts naturally on $V_{k+1}$ by precomposition on polynomials: 
$$
\begin{array}{ccl}
 \SL_2(\CC)\times V_{k+1} & \to & V_{k+1} \\
 (A, P) & \mapsto & P\circ A^t
\end{array}
$$
We consider the transpose matrix $A^t$ so that the action is on the left.
We could also have considered the inverse instead of the transpose, as there exist a matrix $B\in \SL_2(\CC)$
such that $A^{-1}=B A^t B^{-1}$ for every $A\in \SL_2(\CC)$.
This defines a representation 
$$
\operatorname{Sym}^{k}\colon  \SL_2(\CC) \to  \SL_{k+1}(\CC). 
$$
Not only  $\operatorname{Sym}^{k}$  has determinant $1$, but preserves a non-degenerate bilinear form that is 
symmetric for $k$ even, and skew-symmetric for $k$ odd.
So
$$
\operatorname{Sym}^{k}\colon  \SL_2(\CC) \to  
\left\{
\begin{array}{ll}
 \operatorname{SO} (2l+1,\mathbb{C}) & \textrm{for } k=2l \\
 \operatorname{Sp} (2l+2,\mathbb{C}) & \textrm{for } k=2l +1 \\
\end{array}
\right.
$$
For $k=1$,  $\operatorname{Sym}$ is the standard representation and the bilinear form is the determinant
of the $2\times 2$ matrix obtained from two vectors. Then
the form invariant for $\operatorname{Sym}^k$ is the $k$-th symmetric product 
of this determinant.

We shall also consider the complex conjugate $ \overline{  \operatorname{Sym}^{k}}$, which is antiholomorphic.

\begin{prop} Classification of irreducible representations of $\SL_2(\CC)$, cf.~\cite{BW00}.
\begin{enumerate}[(a)]
 \item  Any irreducible holomorphic representation of $\SL_2(\CC)$ is equivalent to $\operatorname{Sym}^{k}$ for a unique $k\geq 0$.
 
\item Any irreducible  representation of $\SL_2(\CC)$ is equivalent to 
 $$\operatorname{Sym}^{k_1,k_2}:= \operatorname{Sym}^{k_1}\otimes \overline{  \operatorname{Sym}^{k_2} } $$
 for a unique pair of integers $k_1,k_2\geq 0$.

\end{enumerate}

\end{prop}

Notice that $\operatorname{Sym}^{k}=\operatorname{Sym}^{k,0}$ and that $\operatorname{Sym}^{0}=\operatorname{Sym}^{0,0}$ is the trivial representation in $\CC$.
The adjoint representation $\operatorname{Ad}\colon\SL_2(\CC)\to 
\operatorname{SO}(\mathfrak{sl}_2(\CC))$ is equivalent to 
$\operatorname{Sym}^2$,
the invariant form being the Killing form.

The representation $\operatorname{Sym}^{k,k}$ takes  values in $\SL_{(k+1)^2}(\RR)$. More precisely, the image 
of $\operatorname{Sym}^{k,k}$ is contained in $\operatorname{SO}(p,q)$, with
$$
p=\frac{k^2+3k+2}{2}\quad\textrm{ and }\quad q= \frac{k^2+k}{2}.
$$
For instance, $\operatorname{Sym}^{1,1}\colon \PSL_2(\CC) \to  
\operatorname{SO}_0(3,1)$ is the canonical isomorphism between 
different presentations of the group of orientation preserving isometries of hyperbolic 3-space.
This representation is used to study infinitesimal deformations of conformally flat structures on a hyperbolic 3-manifold.
The representation $\operatorname{Sym}^{2,2}\colon\PSL_2(\CC) \to  
\operatorname{SO}_0(6,3)$ is used in the study of 
infinitesimal deformations of projective structures on a hyperbolic 3-manifold.

Let $\varrho_\sigma\colon\pi_1 M^3\to \SL_2(\CC)$ denote a lift corresponding to a spin structure $\sigma$, we shall denote its composition with 
$ \operatorname{Sym}^{k_1,k_2}$ by
\begin{equation}
 \label{eqn:varrho}
 \varrho_\sigma^{k_1,k_2}:=    \operatorname{Sym}^{k_1,k_2}\circ        \varrho_\sigma,
\end{equation}
and $ \varrho_\sigma^{k}= \varrho_\sigma^{k,0}$.

\begin{thm}[\cite{Rag65}]
\label{thm:rig} 
 Let $M^3$ be a compact hyperbolic orientable 3-manifold  and 
$\widetilde{\mathrm{hol}}$ a lift of its   holonomy. If $k_1\neq k_2$ then
 $$
 H^*(M^3 ;  \varrho_\sigma^{k_1,k_2} )=0\, .
 $$
\end{thm}

This theorem, due to Raghunathan, is commented in Appendix~\ref{sec:cohomology}, with other facts about cohomology.

\subsection{Higher torsions for closed manifolds}
 \label{sect:higher}

From Theorem~\ref{thm:rig}, we can define torsions.

\begin{definition}
Let $M^3$ be a closed oriented hyperbolic three-manifold, with spin structure $\sigma$, 
and let $\varrho_\sigma^{k_1,k_2}=    \operatorname{Sym}^{k_1,k_2}\circ        \varrho_\sigma$ be as in \eqref{eqn:varrho}.

For $k_1\neq k_2$, define
  \begin{equation}
  \label{eqn:taukk}
  \tau^{k_1,k_2}(M^3,\sigma):=\tor(M^3,  \varrho_\sigma^{k_1,k_2} )  \in 
\mathbb{C}^*/{\pm 1}\, ,
  \end{equation}
and   for $k\geq 1$
\begin{equation}
  \label{eqn:tauk}
   \tau^{k}(M^3,\sigma):= \tau^{k,0}(M^3,\sigma)= \tor(M^3,  \varrho_\sigma^{k})   .
  \end{equation}
\end{definition}
When $k=1$ one obtains the torsion of Definition~\ref{defn:tor2closed}:
$$
\tau^1(M^3,\sigma)=\tau^{1,0}(M^3,\sigma)=\tau(M^3,\sigma).
$$

\begin{remark}
\begin{enumerate}
 \item For $k_1$ or $k_2$ odd, there is no sign indeterminacy as it is a 
representation in $\mathbb{C}^{(k_1+1)(k_2+1)}$, which is even dimensional 
(Remark~\ref{rem:sign}). 
 
 \item   For $k_1+k_2$ even, $\operatorname{Sym}^{k_1}\otimes \overline{\operatorname{Sym}^{k_2}} $ factors through $\operatorname{PSL}_2(\mathbb{C})$, 
 hence it does not depend on the spin structure. In addition,  the 
 sign indeterminacy can be avoided by  using Turaev's refined torsion and an orientation in homology, provided by Poincaré duality \cite{Turaev01,Turaev02}.
 
 \item By construction, 
$\tau^{k_1,k_2}(M^3,\sigma)=\overline{\tau^{k_2,k_1}(M^3,\sigma)}$. 

\item When $k_1=k_2$, $\varrho^{k,k}$ may be not acyclic (for instance
 when it contains a totally geodesic surface \cite{Millson85}), but sometimes it can be acyclic (e.g.\ almost all
 Dehn fillings on two bridge knots for $k=1$ by \cite{Kapovich94, Scannell02, 
StefanoP08} or on the figure eight knot for $k=2$ by 
\cite{Heusener-Porti2011}).
 When   $\varrho^{k,k}$ is acyclic, then  $\tau^{k,k}(M^3,\sigma)$ is also well defined.
 
 \item Changing the orientation. Let $M^3$ and $\overline M^3$ denote the same manifold with opposite orientations and let $\sigma$
 and $\overline\sigma$ denote the corresponding spin structures as in Proposition~\ref{prop:spin}. Then
 $$
 \tau^{k_1,k_2}(\overline M^3, \overline \sigma)= (-1)^{b_1(M^3) (k_1+k_2)}\, \overline{ \tau^{k_1,k_2}(M^3,\sigma)}. 
 $$
 where $b_1(M^3)$ denotes the first Betti number. 
Hence if $M^3$ is amphicherical, then $\tau^{k_1,k_2}(M^3,\sigma)$ is real for
$b_1(M^3) (k_1+k_2)$ even and purely imaginary for $b_1(M^3) (k_1+k_2)$ odd.
\end{enumerate}
\end{remark}

In \cite{Muller12} W.~M\"uller found a beautiful theorem about the asymptotic behavior when $k\to\infty$:

\begin{thm}[\cite{Muller12}]
\label{thm:Muller}
 Let $M$ be a compact, oriented, hyperbolic 3 manifold with a spin structure $\sigma$. Then
\begin{equation}
\label{eqn:muller}
 \lim_{k\to\infty} \frac{ \log\vert  \tau^{k}(M^3,\sigma)\vert}{k^2} = \frac 
1{4\pi}\operatorname{vol} (M^3)\, .
\end{equation}
 In particular, the volume can be computed from the sequence of torsions. 
 \end{thm}

 Again, with our convention our torsion is the inverse of the usual analytic torsion, so
 Equation~\eqref{eqn:muller} has opposite sign to  \cite{Muller12}.

 As noticed in  \cite{Muller12}, this theorem implies that there are finitely many manifolds with the same sequence of torsions. 
 
\begin{question}
Is there a finite collection of integers $0<k_1<\cdots<k_n$ such that  there are finitely many manifolds
with the same $k_1,\ldots, k_n$-torsions?  
\end{question}

The answer is yes under the extra hypothesis that the volume is bounded above.
More precisely, an infinite sequence of manifolds
with bounded volume accumulates (for the geometric topology) to a finite set of cusped manifolds with finite volume,
this is J\o{}rgensen-Thurston properness of the volume function \cite{ThurstonBullAMS, ThurstonBook}.
For such a sequence of manifolds $M_n^3$, $\tau^2(M_n^3,\sigma)\to \infty$, by Corollary~\ref{cor:denseh}, see also \cite[Corollaire~4.18]{Memoirs}. 

A priory, for some $N\in \NN$ there could exists an infinite sequence of hyperbolic manifolds whose volume goes to infinity and 
with the same $k$-torsions for $k\leq N$. By Theorem~\ref{thm:BV} the injectivity radius of this sequence would not converge to infinity.

M\"uller proves  Theorem \ref{thm:Muller} using analytic torsion and Ruelle zeta functions. There is first a theorem of Wotzke \cite{Wot08} that generalizes Fried's theorem
 (Theorem~\ref{thm:Fried}) and from this he obtains a functional equation that relates the volume, the Ruelle zeta functions and the torsions.
 Theorem~\ref{thm:Muller} has been generalized by M\"uller and Pfaff \cite{MullerPfaff14} by working directly on the trace of the heat kernel in \eqref{eqn:traceheatkernel}.

\begin{thm}[\cite{MullerPfaff14}]
 Let $M^3$ be a compact, oriented, hyperbolic 3-manifold with a spin structure $\sigma$. Then
\begin{equation}
 \lim_{k_1\to\infty} \frac{ \log\vert \tau^{k_1,k_2}(M^3,\sigma)\vert}{k_1^2} = 
\frac 1{4\pi}(k_2+1)\operatorname{vol} (M^3)\, .
\end{equation}
 \end{thm}
 
 Again, the sign convention here differs from \cite{MullerPfaff14}.

 \subsection{Functions on the variety of characters}
 
 In this subsection $M^3$ denotes an orientable hyperbolic 3-manifold of  finite volume, with a single cusp. 
 The subsection discusses the torsion as function on the variety of characters.

 \subsubsection{Generic cohomology on the distinguished component}
 \label{sec:genericfunctions}

Some cohomology preliminaries are required before  defining the torsion 
 of $\operatorname{Sym}^{k_1,k_2}$ as function on $X_0(M^3)$. 
 To simplify, use the notation
$$
\rho^{k_1,k_2}:=\operatorname{Sym}^{k_1,k_2}\circ\rho
$$
for any representation $\rho\colon\pi_1M^3\to \SL_2(\CC)$.

 \begin{definition}
 \label{def:exceptional}
A character $\chi\in X_0(M^3)$ is called $(k_1,k_2)$-\emph{exceptional} if
there exists a representation $\rho\in\hom(\pi_1M^3,\SL_2(\CC))$ with character
$\chi_\rho=\chi$ such that $H^0(M^3;\rho^{k_1,k_2})\neq 0$.

The set of $(k_1,k_2)$-exceptional characters is denoted by $\EE^{k_1,k_2}$.
\end{definition}

If a character is irreducible, then the set of representations with this character is
precisely a conjugation orbit, but for reducible characters there may be representations
with the same characters that are not conjugate.

Notice that in Proposition~\ref{prop:rational} 
we have shown that $\EE^{1,0}$ consist of the trivial character, when
it belongs to $X_0(M^3)$ (hence $\EE^{1,0}= \emptyset$ when $b_1(M^3)=1$).

\begin{lemma}
 \label{lem:exceptionalclosed}
  If $k_1\neq k_2$, then a  $(k_1,k_2)$-exceptional character is reducible. In particular
 the $(k_1,k_2)$-exceptional set $\EE^{k_1,k_2}$ is a finite subset of $X_0(M^3)$.  
\end{lemma}

The proof of this lemma is provided in Appendix~\ref{sec:cohomology} (Lemma~\ref{lem:exceptionalclosedA}).

 For further results, the following definition is convenient.
 
 \begin{definition}
  A subset of $X_0(M^3)$ is \emph{generic} if it is  non empty and Zariski open.
  Here the ground field is assumed to be $\CC$ when working with an 
holomorphic representation  $\operatorname{Sym}^{k}$, and  $\RR$ for
$\operatorname{Sym}^{k_1,k_2}$ when $k_2\neq 0$.
 \end{definition}

For the generic behavior of other cohomology groups, one must  distinguish the case where $k_1$ or $k_2$ is odd from the case  
they are both even. All proofs  are postponed to Appendix~\ref{sec:cohomology}.

 \begin{prop} Let $M^3$ be orientable, hyperbolic and with one cusp. Assume $k_1\neq k_2$.
 \label{prop:generic}
 \begin{enumerate}[(a)]
  \item If $k_1$ or $k_2$ is odd, then
    $$
  \left\{ \chi_\rho\in X_0(M^3) \mid  \rho^{k_1,k_2}\textrm{ is acyclic}\right\}
  $$
    is generic.
  \item If both $k_1$ and $k_2$ are even, then 
   \begin{equation*}
   \left\{ \chi_\rho\in X_0(M^3) \mid \dim H_i(M^3 ; \rho^{k_1,k_2})=1 \textrm{ 
   for } i=1, 2 
   \textrm{ and 0 otherwise} \right\}  
   \end{equation*}
    is generic.
 \end{enumerate}
 \end{prop}

When $k_1$ and $k_2$ are even, we need to discuss the choice of basis in homology and perhaps we 
 still need to remove a Zarsiki closed subset.
 We shall consider, for $k_1$ and $k_2$ even,
\begin{equation}
 \label{eqn:FF}
  \FF^{k_1,k_2}= \left\{ \chi_\rho\in X_0(M^3) \mid \dim H^0(T^2 ; 
\rho^{k_1,k_2})\neq 1\right\}\cup \EE^{k_1,k_2} \, ,
\end{equation}
which is  Zarsiki closed subset of $X_0(M^3)$ (over $\CC$ for $k_2=0$, over $\RR$ otherwise).
In particular it is finite when $k_2=0$ or $k_1=0$.

 \begin{lemma}
 \label{lemma:generictorus} 
 Let $T^2$ denote the peripheral torus. If both $k_1$ and $k_2$ are even, then 
 $  \FF^{k_1,k_2}$ is a Zarsiki closed subset of $X_0(M^3)$ (over $\CC$ for $k_2=0$ or $k_1=0$, 
 over $\RR$ otherwise).
In particular it is finite when $k_2=0$ or $k_1=0$. 
\end{lemma}

Consider the cap product
$$
\cap\colon H^0(T^2; \rho^{k_1,k_2})\times H_i(T^2; \CC)\to H_i(T^2; 
\rho^{k_1,k_2}).
$$
It can be described as follows. Choose $a\in H^0(T^2; \rho^{k_1,k_2})$, $a\neq 
0$, and view it as an element of $\CC^{(n_1+1)(n_2+1)}$ invariant by
the action of $\rho^{k_1,k_2} (\pi_1 T^2)$. Any element in $H_i(T^2; \CC)$ is represented by a simplicial cycle: $z\in C_*(K; \CC)$ with $\partial z=0$
(here $K$ is a triangulation of $T^2$). Choose a lift of $z$ to the universal 
covering, $\tilde z\in C_*(\widetilde K; \CC)$, then
$a\otimes \tilde z\in C_i(K;\rho^{k_1,k_2})$ is a cocycle and 
\begin{equation}
\label{eqn:cap}
a\cap [z]= [a\otimes \tilde z] ,
\end{equation}
where the brackets  denote the class in homology. 

In the next proposition $[T^2]\in H_2(T^2; \ZZ)$ denotes a fundamental class,  $\langle \cdot \rangle$   the linear span,
and  $i\colon T^2\to M^3$ the inclusion.

 \begin{prop}
  \label{prop:genericbasis} 
Let $k_1\neq k_2$ be even integers. For $0\neq a\in  H^0(T^2; \rho^{k_1,k_2})$ 
and $0\neq \gamma\in H_1(T^2; \CC)$, the set
\begin{multline*}
  \left\{ \chi_\rho\in X_0(M^3) \mid \langle i_*(a\cap \gamma ) \rangle = 
H_1(M^3;    \rho^{k_1,k_2})\textrm{ and }   \right.            
   \\
\left.    \langle i_*(a\cap [T^2]) \rangle = H_2(M^3;    \rho^{k_1,k_2})
   \right\} 
   \end{multline*}
    is generic.
 \end{prop}

 \subsubsection{Higher torsions on the variety of characters}

 Again assume that $M^3$ is an oriented hyperbolic 3-manifold of finite volume with a single cusp.
 We define the Reidemeister torsion on the distinguished curve of characters
 $
 X_0(M^3)
 $.
 
 \begin{definition}
 For $k_1$ or $k_2$  odd, with $k_1\neq k_2$, and for $\chi_\rho\in X_0(M^3)- \EE^{k_1,k_2}$,
 where $\EE^{k_1,k_2}$ is as in Definition~\ref{def:exceptional},
 define
 $$
 \TT_M^{k_1,k_2}(\chi_\rho)=
 \begin{cases}
\tor(M^3, \rho^{k_1,k_2}) 	& \textrm{if } \rho^{k_1,k_2} \textrm{ is acyclic;} \\ 
   0				& \textrm{otherwise.}
 \end{cases}
$$
  For $k_1$ and $ k_2$ even, with $k_1\neq k_2$,  $\gamma$ a peripheral curve, and  $\chi_\rho\in X_0(M^3)- \FF^{k_1,k_2}$
(where $ \FF^{k_1,k_2}$ is as in \eqref{eqn:FF}), define
 $$
\TT_{M,\gamma}^{k_1,k_2}(\chi_\rho)=
\begin{cases} \tor(M^3, \rho^{k_1,k_2}, \{a\cap \gamma, a\cap [T^2]\}) &  \textrm{if } \langle a\cap \gamma \rangle= H_1(M^3; \rho^{k_1,k_2}),  \\
				& \quad  \langle  a\cap [T^2] \rangle= H_2(M^3; 
				\rho^{k_1,k_2});  \\
   0				& \textrm{otherwise.} 
\end{cases}
 $$
 \end{definition}

We also use the notation 
$$
 \TT_M^{k}=  \TT_M^{k,0} \qquad\textrm{ and }\qquad \TT^{k}_{M, \gamma}=  
\TT_{M,\gamma}^{k,0} 
$$
for $k$ odd or even respectively. 
 
  The proof of Proposition~\ref{prop:rational} with minor changes yields the following:

 \begin{prop}
  For $k_1\neq k_2$,   $\TT_M^{k_1,k_2}$ and $ \TT_{M,\gamma}^{k_1,k_2}$ are  
nonzero (real) analytic functions on $X_0(M^3)- \EE^{k_1,k_2} $ for $k_1$ or 
$k_1$ odd,
  or  $X_0(M^3)- \FF^{k_1,k_2} $ for $k_1$ and $k_2$ even,
  holomorphic when $k_2=0$ and antiholomorphic when $k_1=0$.
 \end{prop}

 When $H^1(M^3;\ZZ)\cong \ZZ$ there is a well defined Alexander polynomial $\Delta(t)$, and we can give conditions for the regularity of 
 $\TT_M^{k}$ and  $\TT_M^{k, \gamma}$ on the whole curve $X_0(M^3)$.
 
 \begin{prop}
  \label{prop:koddEemplty}
  Assume that $H^1(M^3;\ZZ)\cong \ZZ$ and that no root of $ \Delta(t)$ is a root of unity. Then  for $k$ odd $\EE^{k,0}=\emptyset$.
 Hence  for $k$ odd  $ \TT_{M}^{k}$ is a holomorphic function
   on $X_0(M^3)$.
\end{prop}
   
  \begin{proof}
 The proof of Proposition~\ref{prop:rational} applies here with a minor modification. Notice that the condition on the homology implies that the trivial character
 does not lie in $X_0(M^3)$. On the other hand, the reducible characters  in $X_0(M^3)$ are the characters of the composition of a surjection
 $$ \pi_1 M^3\to \ZZ $$ with a representation that maps the generator of the cyclic group to 
 $$
A= \begin{pmatrix}
  \lambda & 0 \\
  0 & \lambda^{-1}
 \end{pmatrix}
 $$
with $\Delta(\lambda^2)=0$ \cite{Burde1967,dR1967}.
Since $\lambda$ is not a root of unity  and $k$ is odd, $\operatorname{Sym}^k(A)$ has no eigenvalue equal to one, hence $H^0(M^3;\rho^k)=0$ for every
$\rho$ with character in $ X_0(M^3)$. 
\end{proof}

\begin{prop}  Assume that $M^3=S^3-\mathcal{K}$ is a knot exterior in the sphere. Assume also 
that no root of $ \Delta(t)$ is a root of unity and all roots have multiplicity one.
Then for $k$ even,  $\TT^k_{M, \gamma}$ can be extended to the whole curve $X_0(M^3)$.
  \end{prop}

 \begin{proof}
There are two issues to discus: the dimension of $H^0(T^2;\rho)$ and the 
vanishing or not of $H^0(M^3;\rho)$.
For the first one, the assumption $M^3=S^3-\mathcal{K}$ implies that $\rho$ is 
never trivial on the peripheral torus $T^2$. 
It my happen that $\rho(\pi_1T^2)$ is non-trivial but 
$\operatorname{Sym}^{k}(\rho(\pi_1T^2))$
has an invariant subspace of dimension $\geq 2$. However, this issue only 
occurs when  $\rho(\pi_1T^2)$ is contained in  a 1-parameter subgroup $G$ of
$\SL_2(\CC)$ conjugate to the group of diagonal matrices. Then 
$\operatorname{Sym}^{k}(G)$ has only one invariant subspace, that varies 
analytically on $\rho$. 
Thus the element $a$ in the expression of the cup products can be chosen to 
depend analytically on $\rho$.
 
 Next discuss $H^0(M^3;\rho)$.
The argument in  Proposition~\ref{prop:koddEemplty} fails in the case $k$ even  because $\operatorname{Sym}^k$ 
has always an invariant subspace, thus for those reducible representations $H^0(M^3;\rho^k)\neq 0$ and  further 
discussion on  reducible representations is required. 
 Namely, at a reducible representation, 
for the corresponding abelian representation $\rho$ as in the proof of Proposition~\ref{prop:koddEemplty}, the hypothesis that $\lambda$ is
not a root of unity implies that $\operatorname{Sym}^k(A)$ has a single eigenvalue equal to one, hence 
 $\dim H^0(M^3;\rho^k)=1$. 
The results  of \cite{HPS2001,HP05}, and the hypothesis that the root is simple, yield that there are,
up to conjugation, two representations $\rho'$ and $\rho''$ that are non abelian and have the same character as $\rho$.
In particular 
$H^0(M^3;(\rho')^{k_1,k_2})\cong H^0(M^3;(\rho'')^{k_1,k_2})=0$. Now, using the arguments of  Proposition~\ref{prop:rational}, we may define the torsion in a Zariski dense subset
 of the component of $\hom(\pi_1M^3, \SL_2(\CC))$ that corresponds to $X_0(M^3)$. It suffices to consider an open neighborhood of $\rho'$ for the usual topology that projects to a neighborhood of 
 $\chi_\rho$   \cite{HPS2001,HP05}, which proves that the torsion on a punctured neighborhood of $\chi_\rho$ extends to $\chi_\rho$. 
\end{proof}

The assertion for the case $k=2$ has been proved by  Yamaguchi in \cite{Yamaguchi07}. More precisely he computes the precise value of $\TT^{2}_{M,\gamma}$ at the reducible character.

 \begin{example}
  Some torsion  functions are computed here for  the figure eight knot exterior. 
  Let $\theta=\theta_m$ denote the evaluation of a character of a meridian: 
  $ \theta_m(\chi)=\chi(m)$ (i.e.\ the trace of a meridian). As mentioned
   in 
  Example~\ref{ex:fig8sl2},
  in \cite{Kitano94b} Kitano computes:
 $$
 \mathbb{T}_M=\mathbb{T}_M^{1} =2-2 {\theta}.
 $$
Further computations with the help of symbolic software yield:
\begin{eqnarray*}
 \mathbb{T}_M^{3} &= &- \left( {\theta}^{2}-2\,{\theta}-2 \right) ^{2}\\
  \mathbb{T}_M^{5} &=&
2\, \left( {\theta}-1 \right)  \left( {\theta}^{8}+2\,{\theta}^{7}-13\,{\theta}^{6}-20\,{\theta}^
{5}+49\,{\theta}^{4}+48\,{\theta}^{3}-33\,{\theta}^{2}-18\,{\theta}-18 \right) \\
 \mathbb{T}_M^{7} &=&
- \left( {\theta}-1 \right) ^{2} \left( 2\,{\theta}^{7}-4\,{\theta}^{6}-21\,{\theta}^{5}+19
\,{\theta}^{4}+57\,{\theta}^{3}+13\,{\theta}^{2}-18\,{\theta}-6 \right) ^{2}\\
 \mathbb{T}_M^{9} &=& 
  2\, \left( {\theta}-1 \right)  \left( {\theta}^{12}+2\,{\theta}^{11}-13\,{\theta}^{10}-13\,{
{\theta} }^{9}+27\,{\theta}^{8}-{\theta}^{7}+95\,{\theta}^{6} +90\,{\theta}^{5}
\right.\\ & & \left. 
-148\,{\theta}^{4}-74\,{{\theta}
}^{3}+61\,{\theta}^{2}+12\,{\theta}-6 \right) ^{2}.
\end{eqnarray*}
Remark that these torsions are functions on the variable $\theta=\theta_m$, the trace of the meridian. This is not always the case,
for instance for $\TT_{M,\gamma}^2$, computed in \cite{Memoirs} (and in Section~\ref{sec:adj}) and described below.
We also computed:
\begin{multline*}
  \mathbb{T}_M^{2,1}=
  13+3\,{\theta}^{4}-3\,  \overline{\theta}  ^{3}-14\,{\theta}^{2}-3\,
  \overline{\theta}   ^{2}+13\,\overline{\theta}+{\theta}^{4}\overline{\theta}
+4\,{\theta}^{2}  \overline{\theta}   ^{3}+2\,{\theta}^{2}  
\overline{\theta}  ^{2} \\ -16\,{\theta}^{2}\overline{\theta}-{\theta}^{4}  
\overline{\theta}   ^{3}+{\theta}^{4}   \overline{\theta}  ^{2}
+
\eta\overline{\eta} \left( {\theta}^{2}\overline{\theta}-{\theta}^{2}-\overline{\theta}-1
 \right) 
\end{multline*}
where $\eta$ is a variable that satisfies 
$\eta^2= ({\theta}^2-1)({\theta}^2-5)$.
This variable can be written in terms of the traces of other elements, see Equation~\eqref{eqn:x1+x2}.
 \end{example}

In subsection~\ref{subsection:eight} other torsions are computed.
If $m$ and $l$ denote respectively the meridian and the longitude, by \eqref{eqn:T2l}, \eqref{eqn:T2m}, \eqref{eqn:t84}:
$$
\TT_{M,l}^2= 5-2\theta^2, \qquad \TT_{M,m}^2= \frac12 \eta, \quad   \textrm{ and }  \quad  \TT_{M,l}^4=8 (2-\theta^2).
$$

\subsection{Evaluation at the holonomy}

\subsubsection{Cohomolgy at the character of the holonomy}
 
In order to  evaluate the torsion functions at the lift of the holonomy,
some results on the cohomology for this representation are required.
Again the proofs are postponed to Appendix~\ref{sec:cohomology}.

 \begin{thm}[Theorem~\ref{prop:dimHM}]
  \label{prop:dimatvarrho}
  Let $\varrho\colon\pi_1 M^3\to\SL_2(\CC)$ be a lift of the holonomy.
  \begin{enumerate}[(a)]
   \item If $k_1+k_2$ is odd, then $\varrho^{k_1,k_2}$ is acyclic.
   \item If $k_1+k_2$ is even, $k_1\neq k_2$,  then
   $$
   \dim H_i(M^3;    \varrho^{k_1,k_2})= \begin{cases}
                                    1 & \textrm{if } i=1,2\, ,\\
                                    0 & \textrm{otherwise.}
                                   \end{cases}
   $$
  \end{enumerate}
 \end{thm}

 Notice that when both $k_1$ and $k_2$ are odd, the homology of $\varrho^{k_1,k_2}$ differs from the generic homology of  $ \operatorname{Sym}^{k_1,k_2}$ on $X_0(M^3)$, there is a discontinuity in the dimension 
 of the cohomology, which is generically trivial. 
 
 We also need to discuss the basis. Let $i\colon T^2\to M^3$ denote the 
inclusion.
 
 \begin{prop}[Corollary~\ref{cor:Ht2} and Proposition~\ref{prop:basis}]
  \label{prop:basis complete}
     Let $\varrho\colon\pi_1 M^3\to\SL_2(\CC)$ be a lift of the holonomy. Assume that $k_1+k_2$ is even.
     \begin{enumerate}[(a)]
      \item $\dim  H^0(T^2;\varrho)=1$.
      \item For $0\neq a\in H^0(T^2;\varrho)$, $\langle i_*(a\cap [T^2]) \rangle = H_2(M^3;    \varrho^{k_1,k_2})$.
      \item For $0\neq a\in H^0(T^2;\varrho)$ and $0\neq\gamma\in H_1(T_2;\ZZ)$:
       \begin{enumerate}[(i)]
    \item $\langle i_*(a\cap \gamma) \rangle = H_1(M^3;    \varrho^{k_1,k_2})$ if $k_1=0$ or $k_2=0$; 
    in addition
    \begin{equation}
     \label{eqn:cuspsh}
     i_*(a\cap \gamma_1)=\cs(\gamma_1,\gamma_2) \, i_*(a\cap \gamma_2)
    \end{equation}
    where $\cs(\gamma_1,\gamma_2) $ is the cusp shape (Definition~\ref{def:cs}).
    \item $ i_*(a\cap \gamma ) = 0$ if $k_1\neq 0$ and $k_2\neq 0$.
   \end{enumerate}
     \end{enumerate}
 \end{prop}

 \subsubsection{Torsion at the characters of the holonomy}
 
 \begin{prop}
 \label{prop:propslift}
 Let $\chi_\varrho$ be the character of a lift $\varrho$ of the holonomy representation
  \begin{enumerate}[(a)]
   \item\label{item:oddnonzero} When  $k_1+ k_2$ is odd,  $\TT_M^{k_1,k_2}(\chi_\varrho)\neq 0$. 
   
   In particular, when $k$ is odd,  $\TT_M^{k}(\chi_\varrho)\neq 0$.
    \item\label{item:k1k2odd} When $k_1\neq k_2$ are both odd, $\TT_{M}^{k_1,k_2}(\chi_\varrho) = 0$.
   \item \label{item:k1k2even} When  $k_1\neq k_2$ are both even  and $k_1,k_2\neq 0$, for $0\neq\gamma\in H_1(T_2;\ZZ)$, we have   $\TT_{M,\gamma}^{k_1,k_2}(\chi_\varrho) =0$.
   \item \label{item:ratiocurves} For $k$ even and $0\neq\gamma\in H_1(T_2;\ZZ)$,  $\TT_{M,\gamma}^{k}(\chi_\varrho) \neq 0$. In addition
   $$
   \frac{\TT_{M,\gamma_2}^{k}(\chi_\varrho)}{\TT_{M,\gamma_1}^{k}(\chi_\varrho)}= \cs(\gamma_2,\gamma_1),
   $$
   where $0\neq\gamma_1,\gamma_2\in H_1(T_2;\ZZ)$.
  \end{enumerate}
 \end{prop}

 \begin{proof}
For item~\eqref{item:oddnonzero}, Theorem~\ref{prop:dimatvarrho} guarantees that the representation is acyclic and therefore $\TT_M^{k_1,k_2}(\chi_\varrho)\neq 0$.
For \eqref{item:k1k2odd}, the same theorem tells that the cohomology of $\varrho^{k_1,k_2}$ is nonzero in dimension 1 and 2. Since it is generically zero, what we have is 
a discontinuity in the dimension of the cohomology groups. On the other hand, the zeroth homology group of $\varrho^{k_1,k_2}$ vanishes, thus 
$\TT_{M}^{k_1,k_2}(\chi_\varrho) = 0$. The proof of \eqref{item:k1k2even} is analogous, using the vanishing of the cup product in dimension one (not two)
of the cap product in Proposition~\ref{prop:basis complete}. The non-vanishing of the cup product of  Proposition~\ref{prop:basis complete} for $k_1$ even and $k_2=0$ 
yields the first part of item \eqref{item:ratiocurves}, the second part follows from Equation~\eqref{eqn:cuspsh}.
\end{proof}

 \begin{thm}[\cite{MFP14}]
  \label{thm:MFP} For any peripheral curve $\gamma$,
  \begin{equation}
  \label{eqn:MFP}
    \lim_{k\to\infty} \frac{ \log\left\vert   \TT_{M,\gamma}^{2k}(\chi_\varrho) \right\vert }{(2 k)^2} = 
    \lim_{k\to\infty} \frac{ \log\left\vert   \TT_{M}^{2k+1}(\chi_\varrho) \right\vert }{(2 k+1)^2} = 
    \frac 1{4\pi}\operatorname{vol} (M^3)\, .
  \end{equation}
 \end{thm}

 Again, the sign convention is the opposite to  \cite{MFP14}.
This theorem is proved from M\"uller's Theorem~\ref{thm:Muller} in the closed case, by using the approximation of the cusped manifolds by closed manifolds obtained by Dehn filling.
Since M\"uller's proof uses Ruelle zeta functions, the key point is to understand geodesics (of bounded lengths) of the closed manifolds that approximate a cusped one.

\subsubsection{Dehn filling} 
 
The following is a generalization to other representations of the formula for Dehn filling in Proposition~\ref{prop:SF}.
In particular the same context and notation is used.

 \begin{prop}
\label{prop:SFh}
Let $M^3$ be as above. For $\vert p\vert +\vert q\vert$ sufficiently large:
\begin{enumerate}[(a)]
 \item For $k_1$ or  $k_2$   odd,
 $$
 \tau^{k_1,k_2}(M_{p/q}^3,\sigma)=  
 {\mathbb{T}_M^{k_1,k_2}(\chi_{p/q,\sigma} )}{ 
\prod_{i_1=0}^{k_1}\prod_{i_2=0}^{k_2}  
 \frac 1{ 1-(\lambda_{p/q})^{2 i_1-k_1} (\overline {\lambda}_{p/q})^{2 i_2-k_2} 
} } \, .
 $$
 \item For $k_1$ and $k_2$ even, $k_1\neq k_2$,
  $$
 \tau^{k_1,k_2}(M_{p/q}^3)=  
 {\mathbb{T}_{M,pm+ql}^{k_1,k_2}(\chi_{p,q} )}{ 
	\prod_{(i_1,i_2)\in \mathcal I}
 \frac 1{ 1-(\lambda_{p/q})^{2 i_1-k_1} (\overline {\lambda}_{p/q})^{2 i_2-k_2} 
} } \, ,
 $$
where $(i_1,i_2)\in \mathcal I$ if $  0\leq i_1 \leq {k_1}$, $0\leq i_2\leq {k_2}$, and $i_1\neq \frac{k_1}{2}$ or $i_2\neq \frac{k_2}{2}$.
\end{enumerate}
\end{prop}

\begin{cor}
 $$
 \lim_{p^2+q^2} \vert \tau^{2k}(M_{p/q}^3,\sigma)\vert =+\infty \, .
 $$
\end{cor}

\begin{proof}
 Use first Proposition~\ref{prop:propslift}~\eqref{item:ratiocurves} . Since 
the character $\chi_{p,q}$ converges to $\chi_\varrho$, 
$\TT_{M,m}^{2k}(\chi_\varrho)\neq 0$, and 
 $\vert\cs(pm+ql ,m )\vert=\vert p+q\,\cs(l,m)\vert\to\infty$, one has
 $$
 \vert {\mathbb{T}_{M,pm+ql}^{2 k}(\chi_{p,q} )}\vert \to\infty \, .
 $$
 As $\vert\lambda_{p/q}\vert\to 1$, the surgery formula in 
Proposition~\ref{prop:SFh} yields the result.
\end{proof}

The same proof as Corollary~\ref{cor:dense} yields:

\begin{cor}
\label{cor:denseh} Given a spin structure, 
The set of modules of the torsions $\vert \tau^{2k_1+1,k_2}(M_{p/q}^3,\sigma)\vert $ obtained by Dehn filling on $M^3$,
so that $\sigma$ can be extended, is dense in $$
\left[ \frac14{\left\vert \tau^{2k_1+1,k_2} (M^3,\sigma)\right\vert},+\infty\right).
$$
\end{cor}

When $k_1+k_2$ is even and $k_1 k_2\neq 0$, then   $\TT^{k_1,k_2}_M(\chi_\varrho)=0$ (for $k_i$ odd) and $\TT^{k_1,k_2}_{M,\gamma}(\chi_\varrho)=0$ 
(for $k_i$ even). In this case,  we cannot get conclusions from the surgery 
formula.

\subsection{Quantum invariants}

 When $(k_1,k_2)=(2,0)$, $\operatorname{Sym}^2$ is  the adjoint representation and its torsion 
occurs in the volume conjecture.

The role of the torsion in the  expansion 
of the path integral is already mentioned in
the work of Witten \cite{Witten89} and Bar-Natan Witten \cite{BNW91}.  Of 
course the work of  Kashaev  \cite{Kashaev97}
and Murakami-Murakami \cite{MM01} play a key role in the conjecture.

For a knot $\mathcal K$, let $J_N(\mathcal K;z)$ denote the colored Jones 
polynomial of $\mathcal K$. If the knot is hyperbolic, 
let $u$ denote Thurston's parameter of the Dehn filling space (Definition~\ref{def:cs}).
Denote the corresponding character by $\chi_u$ and 
let $\operatorname{CS}(\mathcal K,\chi_u)$ denote the $\CC$-valued 
Chern-Simons invariant of a representation with character $\chi_u$, 
namely the real part is  minus the Chern-Simons invariant and 
the imaginary part is the volume of the representation. 
The following version is 
taken from \cite{Murakami13} by H.~Murakami, who attributes it to \cite{GM08,DG11}.

\begin{conj} 
Let $K$ be a hyperbolic knot. There exists a neighborhood $U\subset\CC$ of the 
origin such that for $u\in U- \pi i \QQ$, we have the following asymptotic 
equivalence:
$$
J_n(\mathcal K;e^{\frac{2\pi i+u}{N}}) \, \underset{N\to\infty}{\sim}\, 
\frac{\sqrt{-\pi}}{2\sinh(\frac{u}{2})} \left( {\TT}^{2}_{\mathcal 
K,m}(\chi_{u}) \right)^{-\frac{1}{2}}
\, \left( \frac{N}{2\pi i+u}\right)^{\frac 12} \, e^{\frac{N}{2\pi i+u} 
\operatorname{CS}(\mathcal K,\chi_u)  }
$$
where $m$ denotes the meridian of the knot.
\end{conj}

Again this paper uses a different convention for torsion from  
\cite{Murakami13} and the other references, as it is the opposite to the 
convention for analytic torsion. 
It has been checked for the figure eight knot and $u$ real, $0<  u < 
\log((3+\sqrt{5})/2)$ by Murakami  in   \cite{Murakami13}.
For torus knots, the volume is zero, but not the Chern Simons invariant nor the torsion, and the asymptotic computations of
Dubois-Kashaev \cite{DK07} and Hikami-Murakami~\cite{HM11} support the corresponding conjecture for torus knots.

\medskip

Related to this conjecture, the torsion is also involved in a potential 
function, introduced by Yokota \cite{Yokota03}. From a diagram of the projection 
of a knot $\mathcal K$, in
\cite{OT15} Ohtsuki and Takata define
$\omega_ 2(\mathcal K)^{-1}$ as the modified Hessian of the potential function 
of the  diagram. They justify (formally) that $\sqrt{\omega_ 2(\mathcal K)}$ is 
the term that appears
in the asymptotic development of the Kashaev invariant and therefore they conjecture 
$$
\frac{1}{\omega_ 2(\mathcal K)}= \pm { 2\,i  \, {\TT}^{2}_{\mathcal K,m}(\chi_{\varrho})}.
$$
In \cite{OT15} they prove that the conjecture holds true for two bridge knots.

\medskip

There is also a remarkable contribution of Dimofte and Garoufalidis \cite{DG13}, 
that define an invariant from an ideal triangulation of a knot exterior and
an enhanced Neumann-Zagier datum. 
Enhanced Neumann-Zagier datum means that, besides the complex collection
shape parameters of the ideal hyperbolic tetrahedra,  they use matrices with integer coefficients
that describe how to glue the tetrahedra and a collection of
integers that code a combinatorial flattening (introduced in  \cite{Neumann90} by Neumann
to calculate the Chern-Simons invariant combinatorially).
Form these data, Dimofte and Garoufalidis construct an invariant of the hyperbolic manifold and they check numerically that
it equals ${1}/{{\TT}^{2}_{\mathcal K,m}(\chi_{\varrho})}$ up to some 
constant.

\section{Representations in $\PSL_{n+1}(\CC)$ and the adjoint}
\label{sec:adj}

In \cite{Memoirs} I considered the torsion corresponding to the adjoint representation as a function on the variety of characters 
in $\SL_2(\CC)$, in fact  $\PSL_2(\CC)$. 
Kitayama and Terashima  \cite{KitayamaTerashima15} considered the analog for $\PSL_{n+1}(\CC)$, i.e.~the torsion for the adjoint representation.

Along the section $M^3$ denotes an oriented finite volume hyperbolic 
three-manifold with one cusp, the case with more cusps would require further 
notation,
but here it is essentially the same. Consider the composition of the holonomy with 
$\operatorname{Sym}^{n}$:
$$\varrho^{n}\colon\pi_1M^3\to \PSL_{n+1}(\CC).
$$
As we deal with representations in $\PSL_2(\CC)$, the holonomy does not need to be lifted  to $\SL_2(\CC)$.

\begin{thm}[\cite{MFP12, MFP12b}]
\label{thm:distn+1}
 The character of $\varrho^{n}$ is a smooth point of the variety of characters  $X(M^3,\PSL_{n+1}(\CC))$, of local dimension $n$.
 It is locally parametrized by the symmetric functions on  any peripheral curve.
\end{thm}

\begin{cor} There exists a unique irreducible component of the character variety  $X(M^3,\PSL_{n+1}(\CC))$ that contains the character of 
 $\varrho^{n}$.
\end{cor}

This component is called the \emph{distinguished component} and it is denoted by
\[
 X_0(M^3,\PSL_{n+1}(\CC)).
\]

In the proof of Theorem~\ref{thm:distn+1}, one shows that $\mathfrak{sl}_{n+1}(\CC)^{\Ad \varrho^{n}(\pi_1 T^2)}\cong \CC^n$.
In addition,  if $\{a_1,\ldots, a_{n}\}$ is a basis for 
$\mathfrak{sl}_{n+1}(\CC)^{\Ad \varrho^{n}(\pi_1 T^2)}$, then 
 $$
 \{ a_1\cap [T^2],\ldots, a_{n}\cap [T^2]\}
 $$ is a basis for $H_2(M^3; \Ad\varrho^n)$
and 
$$
\{ a_1\cap [\gamma],\ldots, a_{n}\cap [\gamma]\}
$$ a basis for $H_1(M^3; \Ad\varrho^n)$, here $\gamma$ is a peripheral curve, non-trivial in $H_1(T^2;\ZZ)$.
See \cite{MFP12, MFP12b,KitayamaTerashima15} for details.
For a generic character $\chi_\rho$ in $ X_0(M^3,\PSL_{n+1}(\CC))$,
 $\mathfrak{sl}(n,\CC)^{\Ad  \rho(\pi_1 T^2)}\cong \CC^n$ and a similar 
construction
 applies to find a basis  for $H_i(M^3; Ad\rho)$. Generalizing  the construction of \cite{Memoirs}, Kitayama and Terashima 
 \cite{KitayamaTerashima15} define the torsion. Here we  consider the possibility that the function vanishes, but before  an exceptional set must be removed.
 
 \begin{definition}
 A representation $\rho\colon\pi_1 M^3\to\PSL_{n+1}(\CC)$ is \emph{exceptional} if:
 \begin{enumerate}[(a)]
  \item  $\mathfrak{sl}_{n+1}(\CC)^{\Ad\rho(\pi_1M^3)}\neq 0$ or
\item $\dim \left(\mathfrak{sl}_{n+1}(\CC)^{\Ad\rho(\pi_1T^2)}\right)> n$, for 
the peripheral torus $T^2$.
 \end{enumerate}
  The set of characters of exceptional representations is denoted by 
 $$\EE_{n+1}\subset X_0(M^3,\PSL_{n+1}(\CC) ).$$
 \end{definition}

\begin{lemma}
 The  set  $\EE_{n+1}$  is  Zariski closed in  $X_0(M^3,\PSL_{n+1}(\CC) )$.
\end{lemma}

\begin{proof}
 The defining properties are Zariski closed in $\hom(\pi_1M^3,\PSL_{n+1}(\CC))$ and invariant by conjugation, so the projection 
 to   $X_0(M^3,\PSL_{n+1}(\CC)) $ is a Zariski closed subset, see \cite{VP89} for instance.
\end{proof}

\begin{definition}
 Let $\chi_\rho\in X_0(M^3,\PSL_{n+1}(\CC))-\EE_{n+1}$, choose $\{a_1,\ldots, 
a_{n}\}$ a basis for $\mathfrak{sl}_{n+1}(\CC)^{ \Ad \rho(\pi_1 T^2)}$
 and let
 $\gamma$ be a peripheral curve.
 \begin{enumerate}[(a)]
  \item If $ \{ a_1\cap [\gamma],\ldots, a_{n}\cap [\gamma]\}$ is a basis for  $H_1(M^3;\Ad\rho)$, define
   \begin{multline}
    \T_{(M,\gamma)}(\chi_\rho)= \operatorname{TOR}(M^3, \{ a_1\cap [\gamma],\ldots, a_{n}\cap [\gamma]\}, \\
    \{ a_1\cap [T^2],\ldots, a_{n}\cap [T^2]\});
   \end{multline}
  \item otherwise set  \begin{equation}
			    \T_{(M,\gamma)}(\chi_\rho)=0.
                          \end{equation}
 \end{enumerate}
\end{definition}

This yields a function
\begin{equation}
 \T_{M,\gamma}\colon  X_0(M^3,\PSL_{n+1}(\CC))-E_{n+1} \to \CC\, .
\end{equation}
It can be checked that it is a well defined regular function,
using the ideas of \cite{Memoirs,KitayamaTerashima15} and Proposition~\ref{prop:rational}. For instance, 
when $ \{ a_1\cap [\gamma],\ldots, a_{n}\cap [\gamma]\}$ is a basis for $H_1(M^3;\Ad\rho)$,
then $ \{ a_1\cap [T^2],\ldots, a_{n}\cap [T^2]\}$ is a basis for  $H_2(M^3;\Ad\rho)$
(because $H_2(M^3;\Ad\rho)$ is naturally isomorphic to $H_2(T ^2;\Ad\rho)\cong 
H^0(T ^2 ; \Ad\rho)$).
When there may be confusion, the index $n+1$ may be included in the notation:

\begin{equation}
 \T_{M,\gamma}^{n+1}\colon  X_0(M^3,\PSL_{n+1}(\CC))-E_{n+1} \rightarrow \CC.
\end{equation}

The following proposition relates  $\T_{M,\gamma}$ with the torsion $\TT^{2n}_{(M,\gamma)}$ when we consider 
symmetric powers of representations in $\PSL_2(\CC)$.

\begin{prop}
\label{prop:product}
 For a generic character  $\chi_\rho\in X_0(M^3,\PSL_2(\CC))$, including the holonomy $\chi_\varrho$,
 $$
 \T_{M,\gamma}^{n+1}(\chi_{\rho^n})=\prod_{i=1}^n \TT_{M,\gamma}^{2 i}(\chi_\rho).
 $$ 
 \end{prop}

 \begin{cor}
   \label{cor:difrentdims}
   For a generic character  $\chi_\rho\in X_0(M^3,\PSL_2(\CC))$, including the holonomy $\chi_\varrho$, 
 $$
 \T_{M,\gamma}^{n+1}(\chi_{\rho^n})=\T_{M,\gamma}^{n}(\chi_{\rho^{n-1}})\, \TT_{M,\gamma}^{2 n}(\chi_\rho).
 $$
 \end{cor}

\begin{proof}[Proof of Proposition~\ref{prop:product}]
 The proof is straightforward from Klebsh-Gordan formula:
 \begin{equation}
  \label{eqn:adKG}
  \Ad\circ\operatorname{Sym}^n=\bigoplus_{i=1}^n \operatorname{Sym}^{2i}
 \end{equation}
and multiplicativity of the torsion for sums of representations.
\end{proof}

The next theorem due to Weil \cite{Weil64}, see also \cite{LM85}, gives a nice interpretation to $\T_{M,\gamma}$:

\begin{thm}[\cite{Weil64}]
\label{thm:Weil}
Let $\Gamma$ be a finitely generated group
and $\rho$ 
an irreducible representation with character
 $\chi_\rho\in X(\Gamma, \PSL_{n+1}(\CC) )-\EE_{n+1}$. Then 
 there is a natural isomorphism
 $$
 T^{Zar}_{\chi_\rho} X(\Gamma, \PSL_{n+1}(\CC) )\cong H^1(\Gamma ; \Ad\rho)
 $$
 where $T^{Zar}$ denotes the Zariski tangent space as a scheme.
 
 In particular, if $\phi\colon\Gamma\to\Gamma'$ is a group morphism, then the 
induced map in cohomology corresponds to the tangent map
 $$
 d\phi^* \colon T^{Zar}_{\chi_\rho} X(\Gamma', \PSL_{n+1}(\CC) )\to  
T^{Zar}_{\phi^*\chi_\rho} X(\Gamma, \PSL_{n+1}(\CC) ).
 $$
 \end{thm}

Few comments are in order here. First at all, the condition    that  $\chi_\rho\not\in\EE_{n+1}$ is used to say
that the infinitesimal commutator of $\Ad\rho$ is trivial, i.e.\ $H^0(\pi_1 M^3;\Ad\rho)=0$.
Secondly, 
the variety of characters is not a variety but a scheme: the defining polynomial ideal may be not reduced, thus we must consider the Zariski tangent space of 
the scheme, perhaps not reduced.
Finally, just mention that there are generalizations of this result when $\rho$  is reducible, see \cite{BenAbdelghani2000,BenAbdelghani2002}.

This interpretation has a nice application for surface bundles over the circle, following again \cite{Memoirs,KitayamaTerashima15}.
Assume that $M^3$ is a   bundle over a circle, with fibre a punctured surface 
$\Sigma$ and monodromy $\varphi\colon\Sigma\to \Sigma$, i.e.
$$
M^3=\Sigma\times [0,1]/ (x,1)\sim (\varphi(x),0).
$$
It has a natural epimorphism $\pi_1M^3\to \ZZ$, corresponding to the projection of the fibration $M^3\to S^1$.
The induced map on the  monodromy is denoted by 
$$
\varphi_*\colon X(\Sigma,\PSL_n(\CC))\to X(\Sigma,\PSL_n(\CC))
$$
and the restriction to $\pi_1\Sigma$ of characters  in $X(M^3,PSL_n(\CC))$ 
restrict to the fixed point set of 
$X(\Sigma,\PSL_n(\CC))$. By Weil's theorem, the map induced in 
$H^1(\Sigma;\Ad\rho)$ can be interpreted as the differential
$
d\varphi_*
$.
Thus, using Proposition~\ref{prop:MapTorus}, the twisted polynomial is
$$
 \det (d \varphi_*- t \operatorname{Id}).
$$
Its evaluation at $t=1$ vanishes because $H^*(M^3;\Ad\rho)\neq 0$. The polynomial is divisible by $(t-1)^n$,
corresponding to the invariant curve $\gamma=\partial \Sigma$, as $\dim 
H^1(\gamma; \Ad\rho)=n $. An argument 
using the exact sequences and the basis of homology yields the following result:

\begin{prop}\cite{Memoirs, KitayamaTerashima15}
\label{prop:bundleadj}
 If $M^3$ is a punctured surface bundle, and $\gamma$ the boundary of the fibre. Then 
 $$
 \T_{M,\gamma} =\left.\frac{\det (d \varphi_*- t \operatorname{Id})}{(t-1)^{n}}\right\vert_{t=1}.
 $$
 \end{prop}
 
 This result is very useful for computing  the torsion, it allows to obtain it from the variety of characters or moduli
 spaces without knowing the representation. For instance  Kitayama and Terashima use cluster algebras to compute it
 \cite{KitayamaTerashima15}. In Subsection~\ref{subsection:eight} we recall the method of \cite{Memoirs} to compute it.

 \subsection{Local parameters and change of curve}
 
Consider $\mathfrak{h}\subset \mathfrak{sl}_{n+1}(\CC)$ the Cartan subalgebra of diagonal matrices,
in particular a diagonal matrix in $\PSL_{n+1}(\CC)$ lies in $\exp \mathfrak{h}$.

 For a generic character $\chi\in X_0(M^3,\PSL_{n+1}(\CC))$,  one would like to consider in a 
 neighborhood $U$ of $\chi$: 
 \begin{equation}
    \label{eqn:log}
    \begin{array}{rcl}
      \log_\gamma\colon U\subset X_0(M^3,\PSL_{n+1}(\CC))&\to& \mathfrak{h}\\
      \chi_\rho & \mapsto & \log \rho(\gamma)  
    \end{array}
 \end{equation}
but a priori this may not be well defined. 
Notice that there are indeterminacies due to the action of the Weyl group (permutation of elements in the diagonal)
and indeterminacies due to the complex logarithm. This motivates the following definition.

\begin{definition}
 \label{defn:chamberreg}
 A representation $\rho\in \hom(\pi_1M^3,\PSL_{n+1}(\CC))$ is \emph{chamber regular} is there exists a peripheral element 
 $\gamma\in\pi_1 T^2$ such  that $\rho(\gamma)$  has $n+1$ different eigenvalues. 
\end{definition}

In terms of Lie groups this is a regularity condition: 
$\rho(\pi_1T^2)$ is contained in a Cartan subgroup and  in the interior of the Weyl chamber of $\PSL_{n+1}(\CC)$.

By \cite{MFP12b}, 
the symmetric functions on the eigenvalues of $\rho(\gamma)$
define a local biholomorphism in a neighborhood of $\varrho^{n}$, hence all eigenvalues of 
$\rho(\gamma)$ are different in a Zariski open set. Thus: 

\begin{remark}
The set of chamber regular
characters is a non-empty Zariski open subset of $X_0(M^3,\PSL_2(\CC))$.
 \end{remark}

For a chamber regular character and a peripheral element $\gamma\in\pi_1 T^2$,
there exist a neighborhood $U \subset X_0(M^3,\PSL_{n+1}(\CC)) $ such that the logarithm $\log_\gamma$
as in \eqref{eqn:log} is defined in $U$. Notice that the eigenvalues of the image of $\gamma$ 
do not need to be different, provided that there is an element in the peripheral group with different eigenvalues.

Next consider a nonzero element $a\in \mathfrak{h}$. Using the Killing form, we 
define $a^\star\in\mathfrak{h}$ to be the pairing with $a$.

\begin{lemma}
 \label{lemma:pairing}
 Let $\rho$ be a chamber regular representation, $\gamma$ a peripheral curve and $a\in\mathfrak{h}$.
 Viewing $H_1(M^3;\Ad\rho)$ as cotangent space:
\begin{equation}
 \label{eqn:cotangent}
 a\cap[\gamma] = d ( a^\star\log_\gamma).
\end{equation}
\end{lemma}

This must be compared with \cite[Lemma~3.20]{Memoirs}.
Before proving this lemma, let us discuss its consequences.

\begin{prop}
 \label{prop:localparameter}
 Let $\chi_\rho\in X_0(M^3,\PSL_{n+1}(\CC))$ be a chamber regular character and $\gamma$ a peripheral curve. Then
 $\T_{M,\gamma}(\chi_\rho)\neq 0$ if and only if $\chi_\rho$ is a scheme reduced smooth point of $X_0(M^3,\PSL_{n+1}(\CC))$
 and $\log_{\gamma}$ is a local parameter.
\end{prop}

The proof follows easily from Lemma~\ref{lemma:pairing}, and we just sketch it. 
Namely, the condition of being scheme reduced and smooth is equivalent to saying that $ \dim H^1(M^3;\Ad\rho)=n$. 
By the standard arguments of the long exact sequence of the pair $(M^3, T^2)$  this implies that $ \dim H^0(T^2;\Ad\rho)=n$
and that $\{a_1\cap [T^2],\ldots,a_n\cap [T^2]\}$ is a basis for $H^2(M^3;\Ad\rho)$.
Then  Lemma~\ref{lemma:pairing} yields  
that $\{a_1\cap [\gamma],\ldots,a_n\cap [\gamma]\}$ is a basis for $H^1(M^3;\Ad\rho)$
iff $\log_{\gamma}$ is a local parameter.

For a chamber regular character, the condition of Proposition~\ref{prop:localparameter}
is equivalent to the notion of $\gamma$-regularity of
\cite[Definition~3.2]{KitayamaTerashima15}.

Notice finally that for any non-trivial  peripheral curve $\gamma$,   for the lift of the holonomy $\varrho$, even if it is not 
chamber regular,
$\T_{M,\gamma}(\chi_{\varrho^n})\neq 0$, by \cite[Theorem~3.4]{KitayamaTerashima15}.

\begin{prop}
\label{prop:change}
Let $\gamma_1, \gamma_ 2$ be two peripheral elements.  In a Zariski open domain in $X_0(M^3,\PSL_{n+1}(\CC))$
$$
\frac{ \T_{(M,\gamma_1)} }{  \T_{(M,\gamma_2)} }
= \pm \operatorname{J} (\log_{\gamma_1} \log_{\gamma_2}^{-1} ).
$$ 
\end{prop}

Notice that the Jacobian 
 $
 \operatorname{J} (\log_{\gamma_1} \log_{\gamma_2}^{-1} )
 $
 is well defined generically on the distinguished component $X_0(M^3,\PSL_{n+1}(\CC))$, i.e.\ in a non-empty open Zariski set, as $\log_{\gamma_2}$ is a local parameter in an open set.
 In addition, this Jacobian is independent of the parametrization of $\mathfrak{h}$, as any change of parametrization cancels in the quotient.
The proof of Proposition~\ref{prop:change} is straightforward from Lemma~\ref{lemma:pairing} and the formula of change of 
basis in homology~\eqref{eqn:change}. 

Proposition~\ref{prop:change} does not cover  $\chi_{\varrho^n}$, the character of $\operatorname{Sym}^n$ of the holonomy of the complete hyperbolic structure, as it is not 
chamber regular. However, Corollary~\ref{cor:difrentdims} and Proposition~\ref{prop:propslift}(\ref{item:ratiocurves}) yield:
$$
{\T_{M,\gamma_2}^{n+1}(\chi_{\varrho^n})}=  \cs(\gamma_2,\gamma_1)^n \, 
{\T_{M,\gamma_1}^{n+1}(\chi_{\varrho^n})},
$$
where $\cs(\gamma_2,\gamma_1)$ denotes the cusp shape (Definition~\ref{def:cs}).

\begin{proof}[Proof of Lemma~\ref{lemma:pairing}] Assume that $\rho$ is a generic representation so that $\rho(\gamma)$ is diagonal with different
eigenvalues. Let $\mathfrak{h}$ denote the Cartan algebra of diagonal matrices, 
and choose a non-zero element $a\in\mathfrak{h}$. 
In particular $a\cap\gamma\in H_1(M^3;\Ad\rho)$.
Consider the dual of $a$:
\begin{equation*}
 \begin{array}{rcl}
   a^\star\colon \mathfrak{h} & \to & \CC \\
        h & \mapsto & B(a,h)
 \end{array}
\end{equation*}
where $B$ denotes the Killing form.
Consider also a first order deformation $\rho_t$ of $\rho$, i.e.
\begin{equation}
 \label{eqn:firstorder}
\rho_t =(1+t\, \dot{\rho}+O(t^2))\rho 
 \end{equation}
where $\dot\rho\colon\pi_1M^3\to \mathfrak{sl}_{n+1}(\CC)$ is a group cocycle, that we project to 
$H^1(M^3;\Ad\rho)$.

Consider finally the Kronecker pairing
\begin{equation*}
  \langle\cdot \rangle\colon H^k(M^3;\Ad\rho)\times H_k(M^3;\Ad\rho)\to \CC
\end{equation*}
defined as follows. Let $z\in C^k(M^3;\Ad\rho) $ be a cocycle, with 
$z\colon C_k(\widetilde M^3;\ZZ)\to \mathfrak{sl}_{n+1}(\CC)$,
and let $h\otimes m\in  C_k(M^3;\Ad\rho)$, where $h\in  \mathfrak{sl}_{n+1}(\CC)$ and $m \in C_k(\widetilde M^3;\ZZ)$. At the (co-)chain level,
the Kronecker pairing is
\begin{equation}
\label{eqn:kronecker}
\langle z, h\otimes m\rangle = B(h,z(m)) 
\end{equation}
where $B$ denotes again the Killing form. This induces a non-degenerate pairing between homology and cohomology \cite{Memoirs}.
After all those preliminaries, to establish the lemma one must prove the following equality
\begin{equation}
 \label{eqn:pairing}
 \langle \dot\rho, a\cap[\gamma]\rangle = \left.\frac{d\phantom{t}}{dt} 
a^\star\log\rho_t(\gamma)\right\vert_{t=0} \, .
\end{equation}

To prove  \eqref{eqn:pairing}, start with the group cohomology version of \eqref{eqn:kronecker} in the current context (see \cite{Memoirs}):
\begin{equation}
\label{eqn:kroneckergroup}
\langle \dot\rho , a\cap [\gamma]\rangle = B(a, \dot\rho(\gamma)). 
\end{equation}
From \eqref{eqn:firstorder} evaluated at $\gamma$ and taking logarithms:
\begin{equation}
\label{eqn:loga}
 \log \rho_t(\gamma)  =   \log \rho(\gamma) + t\, \dot{\rho}(\gamma)+ 
O(t^2)) \, .
\end{equation}
Then  \eqref{eqn:pairing} follows from  $a^\star$ applied to \eqref{eqn:loga}, 
then differentiating and applying
\eqref{eqn:kroneckergroup} to the result.
\end{proof}

Lemma~\ref{lemma:pairing} may be related to the results  of Goldman \cite{Goldman86}.

   \begin{example}[Volumes on $X(\mathcal K,SU(2))$]
  \label{ex:dubois}
  
In  \cite{Dubois05,Dubois06} Dubois makes a relevant contribution of torsions as volume form on
the variety of characters of a knot in $\operatorname{SU}(2)$. Here the Cartan algebra has dimension one
and the logarithm of a matrix in $\operatorname{SU}(2)$ is an angle.
Dubois volume form is, up to sign,
$$
\vol_{\tor}= \pm \frac{d\varphi_\gamma } { \T_{m,\gamma}}
$$
where $\varphi_\gamma$ denotes the angle of the representation of the peripheral curve $\gamma$.
Notice that by Proposition~\ref{prop:change}, this form is independent of the choice of the peripheral curve.
In  \cite{Dubois05,Dubois06}, using Turaev's refinement of torsion and a good choice of $\varphi_\gamma$, 
Dubois avoids the sign indeterminacy.
He also shows that it equals to a volume form defined from a Heegaard splitting, \`a la Johnson (Example~\ref{ex:Johnson}).
This is related to the construction of an orientation on the space of representations of Heusener~\cite{Heusener03}.

Proposition~\ref{prop:change} suggests that a similar volume form can be constructed for representations in
$\operatorname{SU}(n+1)$, taking a convenient parametrization of the Cartan algebra (i.e.~infinitesimal angles).
It remains to know also whether the variety of characters of a knot in $\operatorname{SU}(n+1)$ is non-empty and $n$-dimensional. 
This holds true for instance for two bridge knots. Another issue is to compute explicitly the variety of  $\operatorname{SU}(n+1)$ 
characters
for a knot.
\end{example}
 
 \subsection{An example}
\label{subsection:eight}
Let me use Proposition~\ref{prop:bundleadj} to compute $\T_{M,\gamma}$ for the figure eight knot
for $\PSL_2(\CC)$ and $\PSL_3(\CC)$. 
 For  $\PSL_2(\CC)$ this is done in \cite{Memoirs}, but I recall it here for completeness.

Let $\Gamma=\pi_1 M$ denote the fundamental group of the figure eight knot 
exterior. The presentation
$$
 \Gamma=\langle   a,b, m \mid m\, a\, m^{-1}= a\, b,\,  m\, b\, m^{-1}= b\, a\, 
b
 \rangle
$$
corresponds to the fact  that it is fibered over the circle, 
with fibre a punctured torus, whose fundamental group is the free group 
$F_2=\langle a, b\mid\rangle$. The element $m$ is also a meridian curve of the knot.
The monodromy $\phi\colon F_2\to F_2$ satisfies 
$$\phi(a)=a\, b \qquad \textrm{ and } \qquad 
\phi(b)=b\, a\, b.
$$

Since $F_ 2$ is the derived subgroup of $\Gamma$, every representation of $\Gamma$ in $\PSL_{n+1}(\CC)$
restricts to a representation of $F_2$ in $\SL_{n+1}(\CC) $  that is fixed by the monodromy $\phi^*$.
The set of fixed characters is denoted by 
$$
X(F_2,\SL_{n+1}(\CC))^{\phi^*}.
$$
The following lemma is  proved in \cite{Memoirs, HMP15} (using that $F_2$ is the commutator of $\Gamma$):

\begin{lemma}
 For $n=1,2$, the restriction map induces an isomorphism
 $$
 \overline{X_{irr}(\Gamma,\PSL_{n+1}(\CC))} \cong 
\overline{X_{irr}(F_2,\SL_{n+1}(\CC))^{\phi^*}},
 $$
 where the subindex $_{irr}$ stands for irreducible characters.
\end{lemma}

\subsubsection{Computations for $X(M^3,\PSL_{2}(\CC))$ }
 
 By Fricke-Klein theorem, the variety of characters $X(F^2,\SL_2(\CC))$ is isomorphic to $\CC^3$. More precisely, defining
 $$
 \begin{array}{ll}
     & \alpha_1(\rho)=\tr(\rho(a)) \, ,\\
     & \alpha_2(\rho)=\tr(\rho(b)) \, , \\
     & \alpha_3(\rho)=\tr(\rho(a\,b))  \, .
 \end{array}
$$ 
Fricke-Klein theorem asserts that  $(\alpha_1,\alpha_2,\alpha_3)$ are global coordinates of $X(F^2,\SL_2(\CC))$.
Using the relations of traces, $\forall  A,B\in\SL_2(\CC)$:
\[
 \begin{array}{ll}
\tr(A^{-1})&=  \tr(A),   \\
\tr(A B)&= \tr(A B^{-1})- \tr(A)\tr(B),
\end{array}
\]
we may deduce:
$$
 \begin{array}{ll}
       \phi^*(\alpha_1)&=\alpha_3 \, ,\\
       \phi^*(\alpha_2)&=\alpha_2 \alpha_3-\alpha_1  \, ,\\
      \phi^*(\alpha_3)&=\alpha_3^2 \alpha_2-\alpha_1 \alpha_3-\alpha_2  \, .
  \end{array}
$$
Thus $  \phi^*(\alpha_i)=\alpha_i$ is equivalent to 
$$
\alpha_3=\alpha_1,\qquad \alpha_1+\alpha_2=\alpha_1\alpha_2 \, .
$$
Then, the torsion polynomial is
\begin{equation}
 \label{eqn:tor82t}
{\det ( d \phi_* - t \operatorname{Id})}= (t-1)(t^2+(1-2 \alpha_1 \alpha_2) t+1).
 \end{equation}
Removing the factor $(t-1)$ and evaluation at $t=1$ yields
\begin{equation}
 \label{eqn:tor82}
\T_{M,l}= 3-2\alpha_1\alpha_2 = 3-2\alpha_1-2\alpha_2
 \end{equation}
Using that the trace of the longitude $l=[a,b]$ satisfies  
$\theta_l=\alpha_1^2+\alpha_2^2-\alpha_1-\alpha_2-2$
\cite{Memoirs},
we get
$\T_{M,l}^2=17+4\theta_l$. If $\theta_m$ denotes the trace of the meridian,
the distinguished component $X_0(\Gamma,\SL_2(\CC))$ is the curve
\begin{equation}
 \label{eqn:X0fig8}
\alpha_1^2+\alpha_1-1= \theta_m ^2\,(\alpha_1-1). 
\end{equation}
(Let me emphasize that $\theta_l$  and $\theta_m$ denote traces and not  
angles,
I made this choice because $\tau$ is already used for torsion.)
To get the variety $\PSL_2(\CC)$ characters instead of $\SL_2(\CC)$ just replace $\theta_m^2$ by a new variable.
From \eqref{eqn:X0fig8} and $x_1+x_2=x_1x_2$ we deduce:
\begin{equation}
 \label{eqn:x1+x2}
\alpha_1+\alpha_2= \theta_m^2-1  \qquad \textrm{ and } 
\qquad \alpha_1-\alpha_2= \pm\sqrt{(\theta_m ^2-5) (\theta_m-1)} \, .
\end{equation}
Thus 
\begin{equation}
 \label{eqn:T2l}
 \T_{M,l}=  \TT_{M,l}^2= 5-2\theta^2_m.
\end{equation}
Proposition~\ref{prop:change}  can be worked out \cite{Memoirs} to yield:
\begin{equation}
  \label{eqn:T2m}
\T_{M,m}=   \TT_{M,m}^2= \pm \frac{\alpha_1-\alpha_2}{2}=\pm \left( \alpha_1+\frac{1-\theta_m^2}{2} \right)=\pm \frac{1}{2}
\sqrt{(\theta_m^2-5)(\theta_m^2-1)}.
\end{equation}

\subsubsection{Computations for $X(M^3,\PSL_{3}(\CC))$ }

By a theorem of Lawton \cite{Lawton0} and Will \cite{Will}:
$X(F_2,\SL_3(\CC))$
is a double branched covering of $\CC^8$ with coordinates
 $$
 \begin{array}{ll}
     & \beta_1(\rho)=\tr(\rho(a))  \\
     & \beta_2(\rho)=\tr(\rho(a^{-1}))  \\
     & \beta_3(\rho)=\tr(\rho(b))  \\
     & \beta_4(\rho)=\tr(\rho(b^{-1}))  \\
     & \beta_5(\rho)=\tr(\rho(a\,b))  \\
     & \beta_6(\rho)=\tr(\rho(b^{-1}a^{-1}))  \\
     & \beta_7(\rho)=\tr(\rho(a^{-1}b))  \\
     & \beta_8(\rho)=\tr(\rho(a\,b^{-1}))  
 \end{array}
$$
and the  trace of  $l= [a,b]$ and its inverse, $\vartheta_{l^{\pm 1}}$, are a degree
two extension of the coordinates $\beta_1,\ldots,\beta_8$, and 
 $\vartheta_{l }$ and  $\vartheta_{l^{- 1}}$ are Galois conjugate.
 
As $\phi(l)=l$, we may work in $\CC^8$.
Following \cite{HMP15}, $\phi^*$ can be computed as
$$
 \begin{array}{ll}
   &  \phi^*(\beta_1)= \beta_5 \\
   &  \phi^*(\beta_2)= \beta_{{6}} \\
   &  \phi^*(\beta_3)=  -\beta_{{1}}\beta_{{4}}+\beta_{{3}}\beta_{{5}}+\beta_{{8}}          \\
   &  \phi^*(\beta_4)=  -\beta_{{2}}\beta_{{3}}+\beta_{{4}}\beta_{{6}}+\beta_{{7}}          \\
   &  \phi^*(\beta_5)=  -\beta_{{1}}\beta_{{4}}\beta_{{5}}+\beta_{{3}}{\beta_{{5}}}^{2}-\beta_{{3}}\beta_{{6}}+\beta_{{5}}\beta_{{8}}+\beta_{{2}}          \\
   &  \phi^*(\beta_6)=  -\beta_{{2}}\beta_{{3}}\beta_{{6}}+\beta_{{4}}{\beta_{{6}}}^{2}-\beta_{{4}}\beta_{{5}}+\beta_{{6}}\beta_{{7}}+\beta_{{1}}          \\
   &  \phi^*(\beta_7)=  \beta_{{3}}          \\
   &  \phi^*(\beta_8)=  \beta_{{4}}.
 \end{array}
$$
Now, setting $\beta_i=\phi^*(\beta_i)$, we get rid of four variables (we are left with 
$\beta_1$, $\beta_2$, $\beta_3$, and $\beta_4$), and we deduce that 
$X(F^2,\SL_3(\CC))^{\phi^*}$ has three components:
\begin{enumerate}[(a)]
 \item $V_0$, with equations  $\beta_1=\beta_2$, $\beta_3=\beta_4$.
 \item $V_1$, with equations $\beta_1=\beta_2=1$,
 \item $V_2$, with equations $\beta_3=\beta_4=1$.
\end{enumerate}
The component $V_0$ is the restriction of the distinguished component
$X_0(M^3,\PSL_{3}(\CC) )$. On this component,
the torsion polynomial is 
\begin{multline}
{\det ( d \phi_* - t \operatorname{Id})} =
{(t-1)^2} \left( t^2 + (2-\beta_{{1}}\beta_{{3}}) t +1 \right)  \\ \times
 \left( 
 t^4 + (  -\beta_{{1}}\beta_{{3}}-2\,\beta_{{1}}-2\,\beta_{{3}} ) t^3  +  
 ( 6\,\beta_{{1}}\beta_{{3}}+2  ) t^2+ (  -\beta_{{1}}\beta_{{3}}-2\,\beta_{{1}}-2\,\beta_{{3}} ) t + 1
\right) \, .
\end{multline}
After getting rid of $(t-1)^2$ and evaluating at $t=1$, we get
\begin{equation}
\label{eqn:TMl}
\T_{M,l}=( 4-\beta_1\beta_3) 4(1-\beta_1)(1-\beta_3) \, .
\end{equation}

Next Corollary~\ref{cor:difrentdims} is used to compute $\TT^4_{M,l}$.
The relation between traces in $\SL_2(\CC)$ and their image in $\SL_3(\CC)$
via $\operatorname{Sym}$ yields
$$
\begin{array}{l}
\beta_1= \beta_2 =\alpha_1^2-1\, ,\\
\beta_3=\beta_4  =\alpha_2^2-1 \, .
\end{array}
$$
Using these identities in \eqref{eqn:TMl} we get
$$
\T_{M,l}\circ\operatorname{Sym}= (3-2\alpha_1\alpha_2) 4 (2-\alpha_1^2)(2-\alpha_2^2).
$$
Thus by applying Corollary~\ref{cor:difrentdims} and \eqref{eqn:tor82}, one gets
\begin{equation}
 \label{eqn:t84}
 \TT^4_{M,l}= 4 (2-\alpha_1^2)(2-\alpha_2^2)=8(1-\alpha_1\alpha_2)=8(2-\theta_m^2).
\end{equation}

\begin{appendices}
 
 \section{Not approximating  the trivial representation}
 \label{sec:notapprox}

 For a manifold $M^3$ there are components of $X(M^3,\PSL_2(\CC))$ that consist of characters of abelian representations.
When $b_1(M^3)=1$ those components are curves, and  their union is denoted by 
$$
 X^{ab}(M^3,\PSL_2(\CC)).
$$
The number of 
components of $X^{ab}(M^3,\PSL_2(\CC))$
depends on the torsion of $H_1(M^3;\ZZ)$. 

Being irreducible is a Zariski open property for a character \cite{CullerShalen83}. On the other hand, every reducible character is also the character 
of an abelian representation. This yields that $X^{ab}(M^3,\PSL_2(\CC))$ are precisely the components consisting only of reducible characters.

 \begin{lemma}
 \label{lem:singlecomponent}
  Assume that $b_1(M^3)=1$. Then the trivial character  belongs to a single irreducible component of $X(M^3,\PSL_2(\CC))$, which is one of the curves of 
  $X^{ab}(M^3,\PSL_2(\CC))$.
  \end{lemma}

 \begin{cor}
 If $b_1(M^3)=1$, then  the trivial character does
  not belong to $X_0(M^3,\PSL_2(\CC))$.
 \end{cor}

 \begin{proof}[Proof of Lemma~\ref{lem:singlecomponent}]
 The proof uses the projection
 \begin{equation}
  \label{eqn:proj}
    \begin{array}{rcl}
     \pi\colon \hom(\pi_1 M^3,\PSL_2(\CC)) & \to & X(M^3,\PSL_2(\CC)) \\
     \rho & \mapsto &\chi_\rho
    \end{array}
  \end{equation}
and the dimension of its fibre. For the trivial character $\chi_0$,   a representation $\rho\in \pi^{-1}(\chi_0)$ is conjugate to
 \begin{equation}
  \label{eqn:rhoab}
  \rho(\gamma)= \pm \begin{pmatrix}
                 1 & h(\gamma) \\
                 0 & 1
                \end{pmatrix},
                \qquad \forall \gamma\in\pi_1M^3,
  \end{equation}
 where $h\colon \pi_1M^3\to\CC$ is a group homomorphism.
 (Not to be confused with the cusp shape of peripheral representations, as $h$ 
is defined in the whole group $\pi_1M^3$.)
Conjugating the representation \eqref{eqn:rhoab} by a diagonal matrix means replacing the 
morphism $h$ by a multiple. Thus, as $b_1(M^3)=1$,  there are only two orbits 
by conjugation in  $\pi^{-1}(\chi_0)$:
the trivial and the non-trivial morphism  $h\colon \pi_1M^3\to\CC$. 
By looking at the dimension of the stabilizers, 
these orbits have dimension either 0
(for $h$ trivial) or 2 (for $h$ non-trivial). Hence the dimension of 
$\pi^{-1}(\chi_0)$ is 2. On the other hand, on components $Y$ of 
$X(M^3,\PSL_2(\CC))$ that contain irreducible representations,
the generic dimension of $\pi^{-1}$ is 3, the dimension of $\PSL_2(\CC)$.
Since this dimension is upper semi-continuous, the trivial character cannot 
belong to an irreducible  component of  $X(M^3,\PSL_2(\CC))$  that contains 
irreducible characters. Hence it must belong to a component whose characters 
are all reducible.
 \end{proof}

 \begin{lemma}
  \label{lemma:acyclicorbit}
Let $\rho_1,\rho_2\in\hom(\pi_1M^3,\SL_2(\CC))$ have the same character. If  $\chi_{\rho_1}=\chi_{\rho_2}$ is nontrivial, 
then $\rho_1$ is acyclic if and only if $\rho_2$ is acyclic.
 \end{lemma}

\begin{proof}
For every character there is a unique closed orbit by conjugation, so that every 
other orbit accumulates to it, see \cite{LM85}. 
So we may assume that the conjugation orbit of $\rho_2$ accumulates to $\rho_1$. 
By semi-continuity, $\rho_1$ acyclic implies
that so is $\rho_2$. Next assume that $\rho_1$ is not acyclic. Up to conjugacy, there exists a 
group homomorphism $\phi \colon \pi_1M^3\to\ZZ$ and $\lambda\in\CC-\{0,\pm 1\}$ 
such that
$$
\rho_1(\gamma)=\begin{pmatrix}  \lambda^{\phi(\gamma)} & 0 \\
				0 & \lambda^{-\phi(\gamma)} 
		\end{pmatrix}
	\quad\textrm{ and }\quad
\rho_2(\gamma)=\begin{pmatrix}  \lambda^{\phi(\gamma)} & f(\gamma) \\
				0 & \lambda^{-\phi(\gamma)} 
		\end{pmatrix}	
$$
for every $\gamma\in\pi_1M^3$. 
Here $f\colon \pi_1M^3\to\CC$ is a crossed morphism: 
$f(\gamma_1\gamma_2)= f(\gamma_1) + \lambda^{2\phi(\gamma)} f(\gamma_2)$,
for all $\gamma_1,\, \gamma_2 \in \pi_1M^3 $.
The homology of $\rho_1$ decomposes as a direct sum of $\pi_1M^3$-modules: 
$\CC\oplus\{0\}$
and $\{0\}\oplus\CC$, and both are nonzero (one is dual from the other). The 
$\rho_2$-module does not decompose, but there is an
exact sequence of $\pi_1M^3$-modules 
$$
0\to \CC\oplus\{0\} \to \CC^2\to \CC\to 0
$$
where $ \pi_1M^3$ acts on  $\CC^2$ via $\rho_2$, and the action on the other 
modules is the same as for $\rho_1$. 
From the corresponding long exact sequence in homology it follows easily that $\rho_2$ is not acyclic.
\end{proof}

 \section{Cohomology on the variety or characters}
 \label{sec:cohomology}
 
The aim of this appendix is to provide references and proofs for the result in cohomology of Section~\ref{sect:higher}.

\subsection{The complete structure}

Let $M^3$ be a hyperbolic orientable 3-manifold  and 
$$
\varrho=\widetilde{\mathrm{hol}}\colon\pi_1M³\to\SL_2(\CC)
$$ a lift of its   holonomy. As before  denote by
\begin{equation}
 \label{eqn:rhok1k2}
 \varrho^{k_1,k_2}:= \operatorname{Sym}^{k_1,k_2}\circ \varrho \colon\pi_1 M^3\to \SL_{(k_1+1)(k_2+1)}(\CC).
\end{equation}
Let 
\begin{equation}
 \label{eqn:Ek1k2}
E_{k_1,k_2}= \widetilde M \times_{\varrho^{k_1,k_2}} \CC^{(k_1+1)(k_2+1)}
 \end{equation}
be the   flat bundle twisted by ${\varrho^{k_1,k_2}}$  as in   \eqref{eqn:flabundle}.
In particular its de Rham cohomology is isomorphic to the simplicial cohomology of $ \varrho^{k_1,k_2}$
by de Rham theorem \eqref{eqn:deRham}.

Choosing a Hermitian metric on the fibre of the bundle $E_{k_1,k_2}\to M^3$ and a Riemannian metric on $M^3$, there is a product on  
$E_{k_1,k_2}$-valued differential forms  
$\Omega^*(M^3; E_{k_1,k_2})$ by integration on $M^3$, that we denote by $\langle\cdot,\cdot\rangle$. In particular it makes sense to talk about
$L^2$-forms, as the forms with finite norm.

\begin{thm} [Raghunathan~\cite{Rag65}]
\label{thm:strong_acyclicity}
For $k_1\neq k_2$, there exists a uniform constant $c_{k_1,k_2} > 0$ with the following property. 
For every hyperbolic orientable 3-manifold $M^3$, and every  differential 
form $\omega\in\Omega^*(M^3; E_{k_1,k_2})$ 
with compact support,
$$
\langle \Delta \omega, \omega\rangle \geq c_{k_1,k_2} \langle \omega,\omega\rangle.
$$
\end{thm}

This property implies \emph{strong acyclicity}, as it yields that the spectrum 
of $\Delta$ is bounded below by 
the uniform constant   $c_{k_1,k_2} > 0$.

\begin{cor}[Theorem~\ref{thm:rig}]
\label{cor:acyclc}
 Let $M^3$ be a closed, orientable, hyperbolic 3-manifold, then ${\varrho^{k_1,k_2}}$ is acyclic for $k_1\neq k_2$.
 
In particular $\varrho^{k}=\operatorname{Sym}^{k}\circ \widetilde{\mathrm{hol}}$ is acyclic for $k\geq 1$.
\end{cor}

\begin{remark}
This corollary does not hold true when $k_1=k_2$. Millson \cite{Millson85}  showed that it fails when $k_1=k_2>0$ and $M^3$ contains a 
totally geodesic embedded surface, by means of bending.
 \end{remark}

 \subsubsection{The finite volume case}
 
We next discuss the consequences in the finite volume case of Theorem~\ref{thm:strong_acyclicity}.
The following corollary does not assume finite volume.

\begin{cor}
\label{cor:notL2}
 Let $M^3$ be an orientable  hyperbolic 3-manifold and 
$\widetilde{\mathrm{hol}}$ a lift of its  holonomy.
 For $k_1\neq k_2$ every closed $L^2$-form in $\Omega^*(M^3; E_{k_1,k_2})$ is exact. In particular every element
 in $H^* (M^3;   \varrho^{k_1,k_2}
)$ is represented by a form that is not $L^2$.
  \end{cor}

In order to apply this corollary, we need first to compute the homology and cohomology of the peripheral torus.
All the information on the dimension is given by $  H^0(T^2;\varrho^{k_1,k_2})$,
which is the set of invariants of the module $\CC^{(k_1+1)(k_2+1)}$ by the action of  
$\varrho^{k_1,k_2}(\pi_1 T^2)$.
  
  The following implies Item (a) of Proposition~\ref{prop:basis complete}.

\begin{lemma}
  \label{lemma:inv}
  Let $M^3$ be as above and let $T^2\subset M^3$ be the peripheral torus.
  The  invariant subspace of the peripheral group is
  $$
  H^0(T^2;\varrho^{k_1,k_2})\cong \left\{
  \begin{array}{ll}
   0 &  \textrm{if } k_1+k_2 \textrm{ is odd,} \\
   \CC & \textrm{if } k_1+k_2 \textrm{ is even.}
  \end{array}
  \right.
  $$
\end{lemma}

\begin{proof}
 The lift of the holonomy restricted to $\pi_1T^2\cong \ZZ^2$ is written as 
 $$
 (n_1,n_2)\mapsto (-1)^{\epsilon(n_1,n_2)} \begin{pmatrix}
                   1 & n_1 + n_2 \cs \\
                   0 & 1
                  \end{pmatrix}
$$
where $\cs\in\CC -\RR $ is the  cusp shape in Definition~\ref{def:cs} and
$\epsilon\colon\ZZ^2\to\ZZ/2\ZZ$ is a surjection (here it is relevant that 
$\epsilon$ is non-trivial).
From the representation, it is straightforward that if $k_1+k_2$ is odd, then there is an element whose
eigenvalues are all $(-1)$, and for $k_1+k_2$ even, the subspace of  invariants is generated by the monomial $X^{k_1+1} \bar X ^{k_2+1}$.
\end{proof}

From Poincar\'e duality we get information on $H^2$, but also on $H^1$ because $\chi(T^2)=0$. We also know the dimension of the homology
groups by duality. Thus we have:

\begin{cor}
\label{cor:Ht2}
 Let $M^3$ and $T^2$ as above. 
 \begin{enumerate}[(a)]
                                \item If $k_1+k_2$ is odd then
                                $$H^*(T^2;\varrho^{k_1,k_2})\cong 
				H_*(T^2;\varrho^{k_1,k_2})=0\, .$$
                                \item If $k_1+k_2$ is even, then 
                                $$ 
                                \dim H^i(T^2;\varrho^{k_1,k_2})=\dim H_i(T^2;\varrho^{k_1,k_2})= 
					  \left\{
					  \begin{array}{ll}
					  1 &  \textrm{for } i=0,2, \\
					  2 &  \textrm{for } i=1, \\
					  0 &  \textrm{otherwise. }
					  \end{array}\right.
                                $$
\end{enumerate}
\end{cor}

\begin{thm}[Theorem~\ref{prop:dimatvarrho}]
\label{prop:dimHM}
 Let $M^3$ be a hyperbolic orientable 3-manifold with a single cusp, and let  $T^2$ be a peripheral torus.
 \begin{enumerate}[(a)]
  \item \label{item:odd} When $k_1+k_2$ is odd, then $H_*(M^3; 
\varrho^{k_1,k_2})=0$.
  \item When $k_1+k_2$ is     even with $k_1\neq k_2$, then
  \begin{enumerate}[(i)]
   \item $H_i(M^3; \varrho^{k_1,k_2} )=0$ for $i\neq 1,2$,
   \item $H_2(M^3; 
   \varrho^{k_1,k_2})\cong H_2(T^2; \varrho^{k_1,k_2})\cong \CC$,
   \item $H_1(M^3; \varrho^{k_1,k_2})\cong  \CC$.
  \end{enumerate} 
 \end{enumerate}
 \end{thm}
 
\begin{proof}
The group $H^0(M^3 ; \varrho^{k_1,k_2})$ vanishes, as this is the subspace of 
fixed vectors, and both $\varrho$ and
$\operatorname{Sym}^{k_1}\otimes \overline{\operatorname{Sym}^{k_2}}$ are irreducible.
By Corollary~\ref{cor:notL2}, the map
 $
 H^i(M^3 ; T^2;\varrho^{k_1,k_2})\to H^i(M^3;\varrho^{k_1,k_2})
 $
vanishes. Thus $ H^i(M^3;\varrho^{k_1,k_2})\to H^i(T^2;\varrho^{k_1,k_2})$ is injective,
by the long exact sequence in cohomology of the pair. With Poincar\' e duality
and duality between homology and cohomology, we get Item~\eqref{item:odd}.
Poincar\'e duality also yields that $H^1(T^2; \varrho^{k_1,k_2})\to H^2(M^3,T^2;  \varrho^{k_1,k_2})$
is surjective. Then the lemma follows from the long exact sequence in cohomology and the duality 
between homology and cohomology.
\end{proof}
  
  \subsubsection{A basis in cohomology}
  
To describe  a basis for $H_i(M^3; \varrho^{k_1,k_2})$ when $i=1,2$, recall the
cap product defined in Equation~\eqref{eqn:cap}.
$$
\cap\colon H^0(T^2; \varrho^{k_1,k_2})\times H_i(T^2; \CC)\to H_i(T^2; 
\varrho^{k_1,k_2}).
$$
Let $i\colon T^2\to M^3$ denote the inclusion, and $i_*\colon 
H_j(T^2;\varrho^{k_1,k_2} )\to H_j(M^3;\varrho^{k_1,k_2} )$,
the induced map.

\begin{prop}
\label{prop:basis}
 Let $M^3$ and $T^2$ be as above, $k_1\neq k_2$, $k_1+k_2$ even. Choose $a\in  
H^0(T^2; \varrho^{k_1,k_2})$, $a\neq 0$. Then:
  \begin{enumerate}[(a)]
   \item\label{item:h2} $i_*(a\cap [T^2])$ is a basis for 
$H_2(M^3;\varrho^{k_1,k_2})$,  where  $[T^2]\in H_2(T^2;\ZZ)$ is a fundamental 
class.
   \item\label{item:h1} $i_*(a\cap [\gamma])$ is a basis 
for $H_1(M^3;\varrho^{k_1,0})$ or $H_1(M^3;\varrho^{0,k_2})$, for any 
$[\gamma]\in H_1(T^2;\ZZ)$, $[\gamma]\neq 0$.
    \item \label{item:csratio} If $\varrho$ denotes a lift of the holonomy of 
the complete structure,  
for any pair of peripheral curves $0\neq [\gamma_1],[\gamma_2] \in H^1(T^2;\ZZ) 
$ and for $a\in H^0(T^2; \varrho^{k,0})$
$$
a\cup [\gamma_2]=\cs(\gamma_2,\gamma_1) a\cap [\gamma_1]
$$
in $H_1(T^2; \rho^{k,0})$.

   \item\label{item:h1dif} When $k_1k_2\neq 0$, the cap product
   $$\cap\colon H^0(T^2;\varrho^{k_1,k_2})\times  H_1(T^2;\CC)\to H_1 
(T^2;\varrho^{k_1,k_2})$$ is the trivial map (i.e.~identically zero).
 \end{enumerate}
\end{prop}

\begin{proof}
For Item~\eqref{item:h2}, $a\cap [T^2]$ is a basis for $H_ 2(T^2; \varrho^{k_1,k_2})$, by Poincar\'e duality and Lemma~\ref{lemma:inv}.
Item~\eqref{item:h1} is proved in \cite{MFP12b} for $k_2=0$. It holds true for $k_1=0$ by complex conjugation.

To prove the other two items, we choose a cell decomposition of the torus from a 
square with opposite edges identified. Namely 
there is a $2$-cell $e^2$ represented by a square, whose sides are two copies of
the $1$-cells: $e_1^1$ and $e_2^1$ (with respective homology classes $[e_1^1]=[{m}]$ and  $[e_2^1]=[{l}]$ respectively) and the vertices are four copies of the $0$-cell.
Choose lifts to the universal covering so that 
\begin{equation}
\label{eqn:partial}
\partial \tilde e^2= (1-l)\tilde e^1_1 + (m-1)\tilde e^2_1, 
\end{equation}
where   $m$ and $l$ generate
$\pi_1 T^2$. We may assume that 
$$
\varrho(m)=
\pm \begin{pmatrix}
 1 & 1 \\
 0 & 1
\end{pmatrix}
\qquad
\textrm{ and }
\qquad
\varrho(l)=
\pm \begin{pmatrix}
 1 & \eta \\
 0 & 1
\end{pmatrix}
$$
where 
$
\eta =\cs(l,m)\in\CC-\RR
$.

We prove now  \eqref{item:csratio}. $\operatorname{Sym}^k$ acts on the space of degree $k$ 
homogeneous polynomials on $X$ and $Y$; the one dimensional space invariant by $ \varrho^{k}=\varrho^{k,0}$
is generated by $a= X^k$ (Lemma~\ref{lemma:inv}). By \eqref{eqn:partial}:
$$
\partial( X^{k-1}Y\otimes \tilde e^2) = -\eta  \, a \otimes \tilde e^1_1 +   a \otimes \tilde e^2_1 \\
$$
which in cohomology translates to  $-\eta\, a\cap [m]+ a\cap [l]=0$. This proves \eqref{item:csratio} for a system of 
generators $\{[m],[l]\}$ of $H_1(T^2;\ZZ)$, and it holds in general by linearity.

We prove finally \eqref{item:h1dif}. 
$\operatorname{Sym}^{k_1}\otimes\overline{\operatorname{Sym}^{k_2} }$ acts on the space of degree $k_1$ homogeneous polynomials on $X$ and $Y$,
multiplied by degree $k_2$ homogeneous polynomials on $\bar  X$ and $\bar  Y$;
the one dimensional space invariant by $ \varrho^{k_1,k_2}$ is generated by $a= 
X^{k_1} \bar  X^{k_2}$.
By \eqref{eqn:partial}:
\begin{align*}
\partial( X^{k_1-1}Y\bar  X^{k_2}\otimes \tilde e^2) &= -\eta  \, a \otimes \tilde e^1_1 +   a \otimes \tilde e^2_1 \\
\partial( X^{k_1}\bar  X^{k_2-1}Y\otimes \tilde e^2) &= -\bar  \eta  \, a \otimes \tilde e^1_1 +   a \otimes \tilde e^2_1 
\end{align*}
Namely, in homology $-\eta  \, a \cap [m] + a\cap [l]=-\bar  \eta  \, a \cap [m] + a\cap [l]=0$. 
As $\eta\not\in \RR$, the claim follows. 
\end{proof}
 
 \subsection{Generic representations in the distinguished component.}
 
 Next comes the proof of genericity results used in Paragraph~\ref{sec:genericfunctions}.
 Recall  the notation
 $$
 \rho^{k_1,k_2}:=\operatorname{Sym}^{k_1,k_2} \rho
 .$$
Recall also that a  character $\chi\in X_0(M^3)$ is called $(k_1,k_2)$-\emph{exceptional} if
there exists a representation $\rho\in\hom(\pi_1M^3,\SL_2(\CC))$ with character
$\chi_\rho=\chi$ such that $H^0(M^3;\rho^{k_1,k_2})\neq 0$.
The set of $(k_1,k_2)$-exceptional characters is denoted by $\EE^{k_1,k_2}$.
 
\begin{lemma}
 \label{lem:exceptionalclosedA}
 If $k_1\neq k_2$, then a  $(k_1,k_2)$-exceptional character is reducible. In particular
 the $(k_1,k_2)$-exceptional set $\EE^{k_1,k_2}$ is a finite subset of $X_0(M^3)$.  
\end{lemma}

 \begin{proof}
In the holomorphic case  ($k_2=0$), assume that $H^0(M^3;\rho^{k})\neq 0$. 
Then there is a non-trivial subspace of $\CC^{k+1}$ fixed by $\operatorname{Sym}^k(\overline{\rho(\pi_1M^3)})$, where 
$ \overline{\rho(\pi_1M^3)}$ denotes the Zariski closure of $\rho(\pi_1M^3)$. Since  $\operatorname{Sym}^k$ is irreducible, this means that 
$ \overline{\rho(\pi_1M^3)}$
is not the full group $\SL_2(\CC)$, which in the holomorphic setting means that $\rho$ is reducible.

When $k_2\neq 0$, one can only work with the real Zariski closure, and the previous argument yields that either $\rho$ is reducible or it is
contained in a real subgroup conjugate to $\PSL_2(\RR)$ or $\operatorname{SU}(2)$. The restriction of 
$\operatorname{Sym}^{k_1}\otimes \overline{\operatorname{Sym}^{k_2} }$ to those real subgroups is equivalent to  
$\operatorname{Sym}^{k_1}\otimes {\operatorname{Sym}^{k_2} }$, which by Klebsh-Gordan formula decomposes as
\begin{equation}
\label{eq:KGExc}
 \operatorname{Sym}^{k_1}\otimes {\operatorname{Sym}^{k_2} }= 
 \operatorname{Sym}^{k_1+k_2} \oplus
 \operatorname{Sym}^{k_1+k_2-2} \oplus
 \cdots \oplus
 \operatorname{Sym}^{\vert k_1-k_2\vert }.
\end{equation}
As $k_1\neq k_2$ the powers of $ \operatorname{Sym}$ in  \eqref{eq:KGExc}  are non-trivial, hence the argument in the holomorphic case applies again.
\end{proof}

 Recall that we say that 
a property is generic when it holds true for a non-empty Zariski open subset of $X_0(M^3)$, and that 
the ground field is either $\CC$ or $\RR$, depending on whether the discussion is in the holomorphic setting, 
for $\operatorname{Sym}^{k}$, or not, for $\operatorname{Sym}^{k_1,k_2}$.

  \begin{lemma}
  \label{lem:genericH0} Let $M^3$ be a hyperbolic manifold with one cusp and let $k_1\neq k_2\in \NN$ be such that $k_1$ or $k_2$
  is odd.
  Then the set 
   $$
\{ \chi_\rho\in X_0(M^3) \mid \dim H^0(T^2;\rho^{k_1,k_2})=0\}
  $$
    is a non-empty Zariski open subset of the curve $X_0(M^3)$.
 \end{lemma}
 
 \begin{proof}
  By upper semi-continuity of the dimension of the cohomology (see the proof of Proposition~\ref{prop:rational}), it suffices to show that  
  $\dim H^0(T^2;\rho^{k_1,k_2})=0$ for some representation $\rho$ with character 
in $X_0(M^3)$. When $k_1$ is odd and $k_2$ is even, or vice-versa,
  then it holds for the lift of the holonomy of the complete structure, by Lemma~\ref{lemma:inv}.
  If both $k_1$ and $k_2$ are odd, consider an orbifold Dehn filling on $M^3$ with filling curve $\gamma$ that consists in adding a solid torus 
  with singular core curve
  with branching order  $n$. By the Dehn filling theorem, for $n$ large enough it is hyperbolic and the restriction $\rho$ of its holonomy has character in $X_0(M^3)$. 
  The holonomy of the curve $\gamma$ is conjugate to
  $$
  \rho(\gamma)= - \begin{pmatrix} e^{\frac{\pi i}n} & 0 \\ 0  &   e^{-\frac{\pi i}n} \end{pmatrix}.
  $$
  As the core of the geodesic has non-trivial length, the complex length of its holonomy has nonzero real part. Thus we can find a peripheral curve $\gamma'$ with
  holonomy
  $$
  \rho(\gamma')=  \begin{pmatrix} {\lambda} & 0 \\ 0  &   \frac1{\lambda} \end{pmatrix}
  $$
  with $\lambda$ not real nor unitary. This yields that $\rho^{k_1,k_2}(\gamma')$ has no eigenvalue $1$, hence $H^0(T^2;\rho^{k_1,k_2})=0$.
 \end{proof}

 Next lemma implies Lemma~\ref{lemma:generictorus}, as $\FF^{k_1,k_2}= (X_0(M^3)-Z^{k_1,k_2} )\cup \EE^{k_1,k_2}$. 
 
 \begin{lemma} 
 \label{lemma:Zariski}
 Let $M^3$ be a hyperbolic manifold with one cusp and let $k_1\neq k_2\in \NN$ be such that $k_1$ and $k_2$ are even. 
 \begin{enumerate}[(a)]
    \item\label{item:dimh0=1gen}  The set of characters
  $$
  Z^{k_1,k_2}=\{ \chi_\rho\in X_0(M^3) \mid \dim H_0(T^2;\rho^{k_1,k_2})=1\}
  $$
  is a non-empty Zariski open subset of the curve $X_0(M^3)$.
  \item\label{item:basisgeneric}
   For any $[\gamma]\in H^1(T^2;\ZZ)$, $[\gamma]\neq 0$, the set 
  $$
  Z^{k_1,k_2}_\gamma=\{\chi_\rho\in Z^{k_1,k_2}\mid i_*(a\cap [\gamma]) \textrm{ is a basis for } H_1(M^3;\rho^{k_1,k_2})\}
  $$
  is   non-empty and  Zariski open.
 
 \end{enumerate}
 \end{lemma}

 \begin{proof}
For \eqref{item:dimh0=1gen}, given any non-trivial peripheral element $\gamma\in\pi_1 T^2$, 
for a generic character $\chi_\rho\in X_0(M^3)$, $\rho(\gamma)$ is a diagonal matrix, with eigenvalues 
different from $\pm 1$ and so that $\rho^{k_1,k_2}(\gamma)$ has only one eigenvalue equal to one.
This property is in fact Zariski open (over $\CC$ when $k_2=0$, over $\RR$ 
otherwise), cf.~the proof of Lemma~\ref{lem:genericH0}.

Next comes the proof of \eqref{item:basisgeneric}. Any nonzero element in $ H^1(T^2;\ZZ)$ is represented by a simple closed curve
$\gamma$. Consider the orbifold $\mathcal O_n$ obtained by Dehn filling $M^3$ along $\gamma$, so that the 
soul of the solid torus is the branching locus, with branching index $n\in\NN$. For $n$ sufficiently large,
$\mathcal O_n$ is hyperbolic, and its holonomy restricts to a representation with character $\chi_\rho\in X_0(M^3)$.
As $\operatorname{Sym}^{k_1}\otimes \overline{\operatorname{Sym}^{k_2}}$ of the holonomy of $\mathcal O_n$ is acyclic,
a Mayer-Vietoris argument yields that $a\cap\gamma$ is a basis for $H^1(M^3;\rho^{k_1,k_2})$. Then the argument 
for a generic $\rho$ follows from semi-continuity.

Notice that for all but finitely many curves  $ \gamma$ we may take a trivial branching index $n=1$ and work with manifolds instead of orbifolds.
However, in order  to consider all filling slopes, one needs to work with orbifolds.
 \end{proof}

 For any character $\chi_\rho$ in $ Z^{k_1,k_2}_\gamma$, by the long exact sequence in cohomology and Poincar\'e duality,
 the inclusion induces an isomorphism, $H_2(M^3;\rho^{k_1,k_2})\cong H_2(T^2;\rho^{k_1,k_2})$. In addition, Poincar\' e
 duality again gives a natural isomorphism $ H_2(T^2;\rho^{k_1,k_2})\cong H^0(T^2;\rho^{k_1,k_2})$ which yields that $a\cap [T^2]$
 is a basis for $ H_2(T^2;\rho^{k_1,k_2})$. Thus Lemma~\ref{lemma:Zariski}(b)  yields Proposition~\ref{prop:genericbasis}.

\end{appendices}

\begin{footnotesize}

\bibliographystyle{plain}

% \bibliography{biblio}

\end{footnotesize}

 \textsc{Departament de Matem\`atiques} 
 
 \textsc{Universitat Aut\`onoma de Barcelona.}

\textsc{08193 Bellaterra, Spain}

and 

\textsc{Barcelona Graduate School of Mathematics (BGSMath) }

{porti@mat.uab.cat}

\end{document}